\documentclass[11pt,a4paper]{article}
\date{}

\usepackage{amssymb}
\usepackage{latexsym}
\usepackage{amsmath,amssymb}
\usepackage{amsthm,enumerate,verbatim}
\usepackage{amsfonts}
\usepackage{graphicx}
\usepackage{xcolor}
\usepackage{subfigure}
\usepackage{multimedia}
\usepackage[title]{appendix}

\graphicspath{{figures/}}

\newcommand{\R}{{\mathbb{R}}}

\newcommand{\BEE}{\begin{equation*}}
\newcommand{\EEE}{\end{equation*}}
\newcommand{\BE}{\begin{equation}}
\newcommand{\EE}{\end{equation}}

\numberwithin{equation}{section}

\newtheorem{definition}{Definition}[section]
\newtheorem{theorem}{Theorem}[section]
\newtheorem{lemma}{Lemma}[section]
\newtheorem{algorithm}{Algorithm}[section]

\newtheorem{proposition}{Proposition}[section]
\newtheorem{remark}{Remark} [section]

\newcommand{\IN}{{\mathbb{N}}}

\def\RR{\hbox{I\kern-.2em\hbox{R}}}
\title{\bf\Large $\ell_{1\text{-}2}$ Regularization for Sparse Optimization: Consistency and Global Convergence}

\author{Yaohua Hu\thanks{School of Mathematical Sciences, Shenzhen University, Shenzhen 518060, P. R. China (mayhhu@szu.edu.cn).},
\quad\quad Hao Wang\thanks{School of Information Science and Technology, ShanghaiTech University, Shanghai 201210, P. R. China (wanghao1@shanghaitech.edu.cn).},
\quad\quad Xiaoqi Yang\thanks{Department of Applied Mathematics, The Hong Kong Polytechnic University, Kowloon, Hong Kong (mayangxq@polyu.edu.hk).} 
}
\begin{document}

\maketitle

\noindent {\bf Abstract:}\quad The $\ell_{1\text{-}2}$ regularization method has a strong sparsity promoting capability in approaching sparse solutions of linear inverse problems and gained successful applications in various mathematics and applied science fields. This paper aims to investigate the consistency theory and global convergent algorithms for the $\ell_{1\text{-}2}$ regularization problem. In the theoretical aspect, we introduce a notion of restricted eigenvalue condition relative to the $\ell_{1\text{-}2}$ penalty, and employ it to establish an oracle property and a recovery bound for the global solution of the $\ell_{1\text{-}2}$ regularization problem. In the algorithmic aspect, we propose two types of iterative thresholding algorithms with the truncation technique and the continuation technique, respectively, to solve the $\ell_{1\text{-}2}$ regularization problem. Moreover, under the assumption of the well-known restricted isometry property, we establish the convergence of the proposed algorithms to the ground true sparse solution within a tolerance relevant to the noise level and the recovery bound. Preliminary numerical results show that our proposed algorithms can approach the ground true sparse solution and significantly enhance the sparsity recovery capability, compared with the popular sparse optimization algorithms in the literature.


\section{Introduction}

Sparse optimization aims to find a sparse solution of an underdetermined linear system:
\begin{equation}\label{eq-LS}
b = Ax + \varepsilon,
\end{equation}
where the sensing matrix $A\in \R^{m\times n}$ with $m\ll n$ is known, the observation $b\in \R^m$ is collected with an unknown noise $\varepsilon\in \R^m$, and the variable $x \in \R^n$ is to be estimated. 
The sparsity of variable $x$ accounts to the number of its nonzero components, denoted by $\|x\|_0$. A popular and practical technique for approaching a sparse solution of \eqref{eq-LS} is the following regularization formulation:
\begin{equation}\label{eq-L0}
\min_{x\in \R^n}\, \|Ax-b\|_2^2+\lambda \|x\|_0,
\end{equation}
where $\lambda>0$ is a regularization parameter providing a tradeoff between data fidelity and sparsity. Unfortunately, due to the nonconvex and combinational natures of the $\ell_0$ norm, it is NP-hard to compute a global/local solution of problem \eqref{eq-L0}; see \cite{Natarajan95}.

To overcome this difficulty in designing numerical algorithms, the following (convex) $\ell_1$ regularization problem has been widely studied:
\[
\min_{x\in \R^n}\; \|Ax-b\|_2^2+\lambda\|x\|_1,
\]
where $\|x\|_1:=\sum_{i=1}^n |x_i|$ is a sparsity promoting norm being as the tightest convex approximation to the $\ell_0$ norm. Benefiting from the convexity property, a great deal of attention has been attracted to explore the theoretical property and develop the numerical algorithms for the $\ell_1$ regularization problem. Its nice statistical property and consistency theory, including the oracle property and recovery bound, have been well studied under several types of regularity conditions on matrix $A$, such as the mutual incoherence property (MIP) \cite{Bunea07}, restricted isometry property (RIP) \cite{CandesTao05}, null space property (NSP) \cite{FoucartRauhut2013} and restricted eigenvalue condition (REC) \cite{Bickel09}. Many exclusive and efficient algorithms have been proposed and developed with convergence guarantee, such as the iterative soft thresholding algorithm (ISTA) \cite{Daubechies04}, proximal gradient algorithm (PGA) \cite{BeckTeboulle09}, alternative direction method of multipliers (ADMM) \cite{YZ11} and block coordinate descent method (BCD) \cite{Wright2015}.
However, the $\ell_1$ regularization method suffers from several frustrations in practical applications, including over-penalizing the large absolute entries of the variable \cite{Fan01} and requiring a large number of measurements when applied to compressive sensing \cite{XuZB12}. 

To overcome these drawbacks in applications of the $\ell_1$ regularization, several nonconvex regularization methods, e.g., the smoothly clipped absolute deviation (SCAD) \cite{Fan01}, minimax concave penalty (MCP) \cite{ZhangMCP2010}, $\ell_p$ penalty with $p\in (0,1)$ \cite{ChenXJ10}, $\ell_{1\text{-}2}$ penalty \cite{YinDL2015} and $\ell_1/\ell_2$ penalty \cite{Esser2013}, have been proposed and developed. It has been shown in the literature \cite{LouJSC2015,YinDL2015} that the contours of the $\ell_{1\text{-}2}$ norm are more close to those of the $\ell_0$ norm than the $\ell_1$ and $\ell_p$ norms, and hence the $\ell_{1\text{-}2}$ regularization problem
\begin{equation}\label{eq-L1-L2-intro}
\min_{x\in \R^n}\, \|Ax-b\|_2^2+\lambda (\|x\|_1- \|x\|_2)
\end{equation}
would likely outperform the $\ell_1$ and $\ell_p$ regularizations in the sense that it admits a significantly stronger sparsity promoting capability and allows to obtain a more sparse solution from fewer linear measurements when sensing matrix $A$ is highly coherent.
The $\ell_{1\text{-}2}$ regularization method has gained successful applications in a wide range of fields, such as point source super-resolution \cite{LouJSC2016}, image restoration \cite{LouSIAM2015} and MRI reconstruction \cite{MaLouSIAM2017}.
Motivated by these significant applications, many efforts have been devoted to the development of theoretical properties and optimization algorithms of the $\ell_{1\text{-}2}$ regularization problem \eqref{eq-L1-L2-intro}. Although the nonconvexity and the non-separability of the $\ell_{1\text{-}2}$ penalty may lead to some difficulties in the theoretical and algorithmic study of the $\ell_{1\text{-}2}$ regularization problem \eqref{eq-L1-L2-intro}, the exact recovery property of the $\ell_{1\text{-}2}$ minimization problem with observation measurements constraint has been established under the assumptions of RIP \cite{YinDL2015} and NSP \cite{GeNSP2018}, respectively.
To the best of our knowledge, there is still no study devoted to investigating the oracle property and consistency theory for the $\ell_{1\text{-}2}$ regularization problem \eqref{eq-L1-L2-intro}.
Many practical algorithms have been developed to solve the $\ell_{1\text{-}2}$ regularization problem \eqref{eq-L1-L2-intro} such as the PGA and ADMM \cite{LouYan2018}, and difference of convex functions algorithm (DCA) \cite{YinDL2015}. However, limited by the nonconvexity of the $\ell_{1\text{-}2}$ regularization, only the convergence to a stationary point of problem \eqref{eq-L1-L2-intro} has been established in the literature, while there is still no theoretical evidence to guarantee the ``global convergence", that is, the convergence to a global minimum or the ground true solution. This is a typical deficiency of nonconvex optimization algorithms.

Recently, under the framework of iterative limited shrinkage thresholding algorithms,  the convergence of the PGA to the approximate global minimum or the ground true solution has been investigated in \cite{HuMP25} for certain nonconvex sparse regularization problems, including the SCAD \cite{Fan01}, MCP \cite{ZhangMCP2010} and $\ell_p$ regularization \cite{ChenXJ10}, under the assumption of RIP \cite{CandesTao05}. However, the framework in \cite{HuMP25} only concentrates on  separable penalties, and cannot be applied to deal with the nonseparable $\ell_{1\text{-}2}$ penalty problem \eqref{eq-L1-L2-intro}.

This paper aims to explore the consistency theory and the global convergence of some optimization algorithms for the $\ell_{1\text{-}2}$ regularization problem \eqref{eq-L1-L2-intro}. In the aspect of consistency theory, we first introduce a notion of restricted eigenvalue condition (REC) relative to the $\ell_{1\text{-}2}$ regularization and investigate its sufficient conditions in terms of classical regularity conditions such as the MIP and RIP; see Proposition \ref{prop-REC-sufficient}. Under the assumption of the REC, we obtain an oracle inequality and a recovery bound of order $\mathcal{O}\left(\lambda^2 (\sqrt{s} + 1)^2\right)$ for any point in a level set of the $\ell_{1\text{-}2}$ regularization problem \eqref{eq-L1-L2-intro} at the ground true solution, including the global optimal solution of problem \eqref{eq-L1-L2-intro}; see Theorem \ref{thm-OI+RB}. In the algorithmic aspect, we introduce a thresholding operator to the $\ell_{1\text{-}2}$ regularization by simplifying the analytical formulation of its proximal mapping \cite{LouYan2018}, which is able to overcome the over-penalization phenomenon of the soft thresholding operator on the components with large absolute magnitudes; see the details in Remark \ref{rem-LTO}. By virtue of the $\ell_{1\text{-}2}$ thresholding operator, we propose one iterative thresholding algorithm with the truncation technique \cite{Blumensath09} (ITAT) and one with the continuation technique \cite{HaleYinZhang08} (ITAC); see Algorithms \ref{alg-ITAT} and \ref{alg-ITAC} respectively. The proposed ITAT and ITAC are of simple formulation and low storage requirement, and thus are extremely efficient for large-scale sparse optimization problems.
More importantly, under the assumption of the RIP, we show that the sequence generated by the ITAT converges to an approximate true sparse solution of \eqref{eq-LS} at a geometric rate within a tolerance relevant to the noise level plus the recovery bound and that the output of the ITAC approaches an approximate true sparse solution of \eqref{eq-LS} within a tolerance proportional to the noise level; see Theorems \ref{thm-Hard} and \ref{thm-ITAC} respectively. Preliminary numerical experiments show that the ITAT and the ITAC have strong sparsity promoting capability and outperform the standard PGA on both accuracy and robustness.

This paper is organized as follows. In Section 2, we establish the consistency theory for the $\ell_{1\text{-}2}$ regularization problem \eqref{eq-L1-L2-intro}. In Section 3, we propose the ITAT and the ITAC for solving problem \eqref{eq-L1-L2-intro} and establish their convergence to an approximate true sparse solution of \eqref{eq-LS} under the assumption of the RIP. In Sections 4, we present preliminary numerical results.

The notations adopted in this paper are standard in the literature. As usual, let lowercase letters $x,y,$ and $z$ denote the vectors, caligraphic letters $\mathcal{I}$, $\mathcal{J}$, and $\mathcal{S}$ denote the index sets. For $x\in \R^n$ and $A\in \R^{m\times n}$, we use $x_i$ and $A_{i}$ to denote the $i$-th component of $x$ and the $i$-th column of $A$, respectively. For $\mathcal{I}\subseteq [n]$, we use $\mathcal{I}^c$ to denote the complement of $\mathcal{I}$, and use $x_{\mathcal{I}}$ and $A_{\mathcal{I}}$ to denote the subvector of $x$ consisting of components $i\in \mathcal{I}$ and the submatrix of $A$ consisting of columns $i\in \mathcal{I}$, respectively.
As usual, we write $[n]:=\{1,\cdots,n\}$ and $x_+:=\max\{x,0\}$, and use $x \odot y := (x_iy_i)_{i=1}^n$ to denote the Hadamard product of vectors $x$ and $y$.
The following relations between the $\ell_1$, $\ell_2$ and $\ell_{\infty}$ norms are well-known
\begin{equation}\label{eq-L1-2}
\|x\|_\infty \le \|x\|_2 \le \|x\|_1 \le \sqrt{n}\|x\|_2 \le n\|x\|_\infty.
\end{equation}

\section{Consistency theory}
Throughout this section, we assume that the observation $b$ is collected from a linear transform $A$ on a ground true $s$-sparse solution $\bar{x}$ with support $\mathcal{S}$:
\begin{equation}\label{eq-true-sol}
b=A\bar{x} \, \mbox{ with } \, \mathcal{S}:=\{i:\bar{x}_{i}\neq 0\} \, \mbox{ and } \, s:=|\mathcal{S}|.
\end{equation}
To approach the sparse solution $\bar{x}$, we investigate the $\ell_{1\text{-}2}$ regularization problem:
\begin{equation}\label{eq-L1-L2}
\min_{x\in \R^n} \,\, \frac12\|Ax-b\|_2^2+\lambda (\|x\|_1- \|x\|_2).
\end{equation}

This section is devoted to exploring the consistency theory of the $\ell_{1\text{-}2}$ regularization problem \eqref{eq-L1-L2}, including the oracle property and the recovery bound. To this end, we first introduce a notion of regularity condition and discuss its property.

\subsection{Regularity conditions}

In the scenario of sparse optimization, certain regularity conditions on the sensing matrix $A$ are required to guarantee the consistency theory of sparse regularization problems; see \cite{Bickel09,Bunea07,HuJMLR17,Meinshausen09} and references therein. For example, the restricted eigenvalue condition (REC) 
was introduced in \cite{Bickel09} to investigate the consistency theory of the $\ell_1$ regularization problem.

Inspired by the idea of REC in \cite{Bickel09}, we introduce a notion of REC relative to the $\ell_{1\text{-}2}$ penalty, so as to guarantee the consistency theory of the $\ell_{1\text{-}2}$ regularization problem \eqref{eq-L1-L2}.
To proceed, let $\bar{x}$ and $x^*$ be the ground true $s$-sparse solution defined by \eqref{eq-true-sol} and a global optimal solution of problem \eqref{eq-L1-L2}, respectively.
Note that the residual $\hat{x}:=x^*-\bar{x}$ always satisfies that
\begin{equation*}\label{eq-RECq}
\|\hat{x}_{\mathcal{S}^c}\|_1 \le \|\hat{x}_{\mathcal{S}}\|_1+ \|\hat{x}\|_2;
\end{equation*}
see the proof of Theorem \ref{thm-OI+RB}. Associated to this, we introduce a notion of REC relative to the $\ell_{1\text{-}2}$ penalty as follows.
It is worth mentioning that when the REC relative to $\ell_{1\text{-}2}$ penalty is similar to the $\ell_{1\text{-}2}$ null space property introduced in \cite{ZhangNSP2021}.
Given $1\le s\le t\ll n$, $x\in \R^n$ and $\mathcal{I}\subseteq [n]$, we use $\mathcal{I}(x;t)$ to denote the index set of the first $t$ largest coordinates in the absolute value of $x$ in $\mathcal{I}^c$.

\begin{definition}\label{def-REC}
The matrix $A\in \R^{m\times n}$ is said to satisfies the restricted eigenvalue condition relative to $\ell_{1\text{-}2}$ penalty and $(s,t)$ $(\ell_{1\text{-}2}$-${\rm REC}(s,t) in short)$ if
\[
\phi(s,t):=\min \, \left\{\frac{\|Ax\|_2}{\|x_{\mathcal{J}}\|_2}:|\mathcal{I}|\le s, \|x_{\mathcal{I}^c}\|_1 \le \|x_{\mathcal{I}}\|_1+ \|x\|_2, \mathcal{J}=\mathcal{I}(x;t)\cup \mathcal{I} \right\}>0.
\]
\end{definition}

Several types of regularity conditions have been proposed to investigate the consistency theory and nice statistical property of sparse regularization problems, recalled as follows.


\begin{definition}[RIC \cite{CandesTao05}] \label{def-RIP}
\begin{enumerate}[{\rm (i)}]
 \item The $s$-restricted isometry constant of $A$ is denoted by $\delta_s$ and defined to be the smallest quantity such that, for
any $x\in \R^n$ and $J\subseteq [n]$ with $|J|\le s$,
\begin{equation}\label{eq-RIC}
(1-\delta_s)\|x_J\|_2^2\le \|Ax_J\|_2^2\le (1+\delta_s)\|x_J\|_2^2.
\end{equation}
 \item The $(s,t)$-restricted orthogonality constant of $A$ is denoted by $\theta_{s,t}$ and defined to be the smallest quantity such
that, for any $x\in \R^n$ and $J,T\subseteq [n]$ with $|J|\le s$, $|T|\le t$ and $J\cap T=\emptyset$,
\begin{equation}\label{eq-ROC}
|\langle Ax_J, Ax_{T}\rangle|\le \theta_{s,t}\|x_J\|_2\|x_{T}\|_2.
\end{equation}
\end{enumerate}
\end{definition}

\begin{definition}[MIC \cite{Donoho2001Uncertainty}]
The mutual incoherence constant of $A$ is defined by
\begin{equation}\label{eq-MIP}
\mu:=\sup\{|A_j^\top A_i|: \forall 1\le i \neq j\le n\}.
\end{equation}
\end{definition}

\begin{definition}[SEC \cite{Donoho2006Most}]
The $s$-sparse minimal eigenvalue and $s$-sparse maximal eigenvalue of $A$ are respectively defined by
\begin{equation}\label{sparse-eigen}
\sigma_{\min}(s):=\min_{\|x\|_0\le s} \sqrt{\frac{x^\top A^\top Ax}{x^\top x}},\quad
\sigma_{\max}(s):=\max_{\|x\|_0\le s} \sqrt{\frac{x^\top A^\top Ax}{x^\top x}}.
\end{equation}
\end{definition}

It is natural to investigate the relationships between the $\ell_{1\text{-}2}$-REC and other types of regularity conditions mentioned above. The following proportion provides some sufficient conditions for the $\ell_{1\text{-}2}$-REC in terms of SEC, RIC and MIC, respectively. The proofs are deferred to the Appendix \ref{App-JRFC}.

\begin{proposition}\label{prop-REC-sufficient}
$A\in \R^{m\times n}$ satisfies the $\ell_{1\text{-}2}$-${\rm REC}(s,t)$ provided that one of the following conditions holds:
\begin{enumerate}[{\rm (i)}]
 \item $(\sqrt{t}-1) \sigma_{\min}(s+t)> (\sqrt{s}+1) \sigma_{\max}(t).$
 \item $(\sqrt{s}+1)\theta_{t,s+t} < (\sqrt{t}-1)(1-\delta_{s+t})$.
 \item $(\sqrt{t}-1) \sigma_{\min}(s+t)> 2\theta_{s+t,1} \sqrt{t} (\sqrt{s}+1)$.
 \item $(\sqrt{t}-1) \sigma_{\min}(s+t) > 2\mu \sqrt{t(s+t)} (\sqrt{s}+1)$.
 \item $\|A_i\|_2=1$ for each $i\in [n]$, and $\mu<\frac{\sqrt{t}-1}{(s+t)(\sqrt{t}-1) + 2\sqrt{t(s+t)}(\sqrt{s}+1)}$.
\end{enumerate}
\end{proposition}

The following lemma is recalled from \cite{HuJMLR17} and will be useful in establishing the recovery bound for problem \eqref{eq-L1-L2} in Theorem \ref{thm-OI+RB}.
Let ${\rm rank}_i(x)$ denote the rank of $|x_i|$ among $\mathcal{I}^c$ (in a decreasing order), and separate $\mathcal{I}^c$ into a series of disjoint index sets $\{\mathcal{I}_k\}$ with each set consisting of $t$ indexes in an increasing order of ${\rm rank}_i(x)$. That is, $r:=\left\lceil \frac{n-|\mathcal{I}|}{t} \right\rceil$ (where $\lceil u\rceil$ denotes the largest integer no greater than $u$) and
\begin{equation}\label{eq-Jk}
\mathcal{I}_k(x;t):= \left\{i\in\mathcal{I}^c:{\rm rank}_i(x)\in\{kt+1,\dots,(k+1)t\}\right\} \quad \mbox{for} \quad k\in [r].
\end{equation}

\begin{lemma}[{\cite[Lemma 7]{HuJMLR17}}]\label{lem-bound-Nc}
Let $x\in \R^{n}$, $\mathcal{I}\subset [n]$ and $p \ge 1$, and write $\mathcal{J}:=\mathcal{I} \cup \mathcal{I}(x;t)$ and $\mathcal{I}_k:=\mathcal{I}_k(x;t)$. Then it holds that
\[
\|x_{\mathcal{J}^c}\|_{p} \le \sum_{k=1}^r\|x_{\mathcal{I}_{k}}\|_{p} \le t^{\frac{1}{p}-1} \|x_{\mathcal{I}^c}\|_1. 
\]
\end{lemma}

\subsection{Oracle property and recovery bound}
In this subsection, we present the consistency theory of the $\ell_{1\text{-}2}$ regularization problem \eqref{eq-L1-L2}, including its oracle property and recovery bound, under the $\ell_{1\text{-}2}$-REC. For the sake of simplicity, let $h(x):= \|Ax-b\|_2^2+\lambda (\|x\|_1- \|x\|_2)$ denote the objective function of problem \eqref{eq-L1-L2} and define its level set by
\begin{equation}\label{eq-level}
{\rm lev}_h(\bar{x}):=\{x\in \R^n: h(x) \le h(\bar x)\}.
\end{equation}

\begin{theorem}\label{thm-OI+RB}
Let $(\bar{x},\mathcal{S},s)$ be defined in \eqref{eq-true-sol}. Let $x^*\in {\rm lev}_h(\bar{x})$, and suppose that $A$ satisfies the $\ell_{1\text{-}2}${\rm -REC}$(s,s)$.
Then the following statements are true.
\begin{enumerate}[{\rm (i)}]
 \item The oracle inequality holds
  \begin{equation}\label{eq-OI}
  \frac12 \|Ax^*-b\|_2^2+\lambda (\|x^*_{\mathcal{S}^c}\|_1-\|x^*_{\mathcal{S}^c}\|_2) \le \frac{2\lambda^2 (\sqrt{s} + 1)^2}{\phi^2(s,s)}.
  \end{equation}
 \item The recovery bound holds
  \begin{equation}\label{eq-RB}
  \|x^*-\bar{x}\|_2^2\le \left(4+\max\left\{\frac{\sqrt{s} + 1}{\sqrt{s}}, \frac{4}{\sqrt{s} - 2}\right\}^2\right) \frac{\lambda^2 (\sqrt{s} + 1)^2}{\phi^4(s,s)}.
  \end{equation}
\end{enumerate}
\end{theorem}
\begin{proof}
(i) Let $x^*\in {\rm lev}_h(\bar{x})$. Then it follows from \eqref{eq-level} and $\bar{x}_{\mathcal{S}^c} = 0$ (cf. \eqref{eq-true-sol}) that
\begin{equation}\label{eq-thm-OI-1}
\frac12\|Ax^*-b\|_2^2+\lambda (\|x^*\|_1-\|x^*\|_2) \le \lambda (\|\bar{x}\|_1-\|\bar{x}\|_2) = \lambda (\|\bar{x}_{\mathcal{S}}\|_1-\|\bar{x}_{\mathcal{S}}\|_2).
\end{equation}
Note that
\begin{equation}\label{eq-thm-OI-0}
\|x^*\|_1 = \|x^*_{\mathcal{S}}\|_1 + \|x^*_{\mathcal{S}^c}\|_1 \quad \mbox{and} \quad \|x^*\|_2 \le \|x^*_{\mathcal{S}}\|_2 + \|x^*_{\mathcal{S}^c}\|_2
\end{equation}
(by the subadditivity of $\ell_2$ norm). Then \eqref{eq-thm-OI-1} is reduced to
\begin{equation}\label{eq-thm-OI-1a}
\frac12\|Ax^*-b\|_2^2 + \lambda (\|x^*_{\mathcal{S}^c}\|_1-\|x^*_{\mathcal{S}^c}\|_2)
\le \lambda (\|\bar{x}_{\mathcal{S}}\|_1-\|x^*_{\mathcal{S}}\|_1) + \lambda (\|x^*_{\mathcal{S}}\|_2-\|\bar{x}_{\mathcal{S}}\|_2).
\end{equation}
Note by the subadditivity of $\ell_p$ norm ($p=1,2$) that
\begin{equation}\label{eq-thm-OI-1an}
\|x^*_{\mathcal{S}}\|_2-\|\bar{x}_{\mathcal{S}}\|_2 \le \|x^*_{\mathcal{S}} - \bar{x}_{\mathcal{S}}\|_2 \quad
\mbox{and} \quad
\|\bar{x}_{\mathcal{S}}\|_1-\|x^*_{\mathcal{S}}\|_1 \le \|x^*_{\mathcal{S}} - \bar{x}_{\mathcal{S}}\|_1 \le \sqrt{s} \|x^*_{\mathcal{S}} - \bar{x}_{\mathcal{S}}\|_2
\end{equation}
(due to \eqref{eq-L1-2}).
Then \eqref{eq-thm-OI-1a} is reduced to
\begin{equation}\label{eq-thm-OI-1b}
\frac12\|Ax^*-A\bar{x}\|_2^2 + \lambda (\|x^*_{\mathcal{S}^c}\|_1-\|x^*_{\mathcal{S}^c}\|_2)
\le \lambda (\sqrt{s} + 1) \|x^*_{\mathcal{S}} - \bar{x}_{\mathcal{S}}\|_2.
\end{equation}
On the other hand, noting by \eqref{eq-true-sol} that $\bar{x}_{\mathcal{S}^c} = 0$, we obtain by \eqref{eq-thm-OI-1an} and the subadditivity of $\ell_2$ norm that
\begin{equation*}\label{eq-thm-OI-1c}
\begin{array}{lll}
\|x^*_{\mathcal{S}^c}-\bar{x}_{\mathcal{S}^c}\|_1 - \|x^*_{\mathcal{S}}-\bar{x}_{\mathcal{S}}\|_1 - \|x^*-\bar{x}\|_2 &\le \|x^*_{\mathcal{S}^c}\|_1 - (\|\bar{x}_{\mathcal{S}}\|_1-\|x^*_{\mathcal{S}}\|_1) - (\|x^*\|_2-\|\bar{x}\|_2)\\
&= \|x^*\|_1-\|x^*\|_2 - (\|\bar{x}\|_1-\|\bar{x}\|_2)\\
&\le 0
\end{array}
\end{equation*}
(thanks to \eqref{eq-thm-OI-0} and \eqref{eq-thm-OI-1}). This shows that $x^*-\bar{x}$ falls in the constraint set of the $\ell_{1\text{-}2}${\rm -REC}$(s,s)$. It thus follows from Definition \ref{def-REC} that
\begin{equation*}\label{eq-thm-OI-2}
\|x_{\mathcal{S}}^*-\bar{x}_{\mathcal{S}}\|_2\le \frac1{\phi(s,s)}\|Ax^*-A\bar{x}\|_2.
\end{equation*}
Thus \eqref{eq-thm-OI-1b} is reduced to
\begin{equation}\label{eq-thm-OI-3}
\frac12\|Ax^*-A\bar{x}\|_2^2 + \lambda (\|x^*_{\mathcal{S}^c}\|_1-\|x^*_{\mathcal{S}^c}\|_2)
\le \frac{\lambda (\sqrt{s} + 1)}{\phi(s,s)} \|Ax^*-A\bar{x}\|_2;
\end{equation}
consequently,
\begin{equation}\label{eq-thm-OI-4}
\|Ax^*-A\bar{x}\|_2\le \frac{2\lambda (\sqrt{s} + 1)}{\phi(s,s)}.
\end{equation}
This, together with \eqref{eq-thm-OI-3}, yields the oracle inequality \eqref{eq-OI}.

(ii) Define ${\mathcal{N}}:=\mathcal{S}\cup \mathcal{S}(x^*;s)$. 
The $\ell_{1\text{-}2}${\rm -REC}$(s,s)$ implies that
\begin{equation}\label{eq-thm-RB-1}
\|x_{\mathcal{N}}^*-\bar{x}_{\mathcal{N}}\|_2 \le \frac{\|Ax^*-A\bar{x}\|_2}{\phi(s,s)} \le \frac{2\lambda (\sqrt{s} + 1)}{\phi^2(s,s)}
\end{equation}
(by \eqref{eq-thm-OI-4}). By Lemma \ref{lem-bound-Nc} (with $2$, $s$, $\mathcal{N}$, $\mathcal{S}$ in place of $p$, $t$, $\mathcal{J}$, $\mathcal{I}$), we get that
\begin{equation}\label{eq-thm-RB-2}
\|x_{\mathcal{N}^c}^*\|_2 \le \frac1{\sqrt{s}} \|x_{\mathcal{S}^c}^*\|_1;
\end{equation}
consequently,
\begin{equation}\label{eq-thm-RB-3}
\|x_{\mathcal{N}^c}^*\|_2 - \frac{1}{\sqrt{s}} \|x_{\mathcal{S}^c}^*\|_2 \le \frac1{\sqrt{s}} (\|x_{\mathcal{S}^c}^*\|_1-\|x_{\mathcal{S}^c}^*\|_2).
\end{equation}
Note by \eqref{eq-thm-OI-3} and the Cauchy inequality $a^2+b^2\ge 2ab$ that
\[
 \|Ax^*-A\bar{x}\|_2 \sqrt{2\lambda (\|x^*_{\mathcal{S}^c}\|_1-\|x^*_{\mathcal{S}^c}\|_2)}
\le \frac{\lambda (\sqrt{s} + 1)}{\phi(s,s)} \|Ax^*-A\bar{x}\|_2;
\]
hence
\[
\|x^*_{\mathcal{S}^c}\|_1-\|x^*_{\mathcal{S}^c}\|_2\le \frac{\lambda (\sqrt{s} + 1)^2}{2\phi^2(s,s)}.
\]
This, together with \eqref{eq-thm-RB-3}, implies that
\begin{equation}\label{eq-thm-RB-4}
\|x_{\mathcal{N}^c}^*\|_2 - \frac{1}{\sqrt{s}} \|x_{\mathcal{S}^c}^*\|_2 \le \frac{\lambda (\sqrt{s} + 1)^2}{2 \sqrt{s} \, \phi^2(s,s)}.
\end{equation}
Below we consider the following two cases.\\
Case 1: Suppose that $\|x_{\mathcal{N}^c}^*\|_2\ge \frac{2}{\sqrt{s}} \|x_{\mathcal{S}^c}^*\|_2$. Then \eqref{eq-thm-RB-4} is reduced to
\begin{equation}\label{eq-thm-RB-5}
\|x_{\mathcal{N}^c}^*\|_2 \le \frac{\lambda (\sqrt{s} + 1)^2}{ \sqrt{s} \, \phi^2(s,s)}.
\end{equation}
Case 2: Suppose that $\|x_{\mathcal{N}^c}^*\|_2< \frac{2}{\sqrt{s}} \|x_{\mathcal{S}^c}^*\|_2$. Noting that ${\mathcal{N}}=\mathcal{S}\cup \mathcal{S}(x^*;s)$ being a disjoint decomposition and $\bar{x}_{\mathcal{S}(x^*;s)} = 0$ as $\mathcal{S}(x^*;s) \subseteq \mathcal{S}^c$, we obtain that
\[
\|x_{\mathcal{S}(x^*;s)}^*\|_2 = \|x_{\mathcal{S}(x^*;s)}^* - \bar{x}_{\mathcal{S}(x^*;s)}\|_2 \le \|x_{\mathcal{N}}^*-\bar{x}_{\mathcal{N}}\|_2 \le \frac{2\lambda (\sqrt{s} + 1)}{\phi^2(s,s)}
\]
(due to \eqref{eq-thm-RB-1}), and then we get by the subadditivity of $\ell_2$ norm that
\[
\|x_{\mathcal{S}^c}^*\|_2 \le \|x_{\mathcal{N}^c}^*\|_2 + \|x_{\mathcal{S}(x^*;s)}^*\|_2 \le \|x_{\mathcal{N}^c}^*\|_2 + \frac{2\lambda (\sqrt{s} + 1)}{\phi^2(s,s)}.
\]
This, together with the assumption of Case 2, implies that
\begin{equation}\label{eq-thm-RB-6}
\|x_{\mathcal{N}^c}^*\|_2 \le \frac{4}{\sqrt{s} - 2} \frac{\lambda (\sqrt{s} + 1)}{\phi^2(s,s)}.
\end{equation}
Combining \eqref{eq-thm-RB-5} and \eqref{eq-thm-RB-6} yields that
\begin{equation*}\label{eq-thm-RB-7}
\|x_{\mathcal{N}^c}^*\|_2 \le \max\left\{\frac{\sqrt{s} + 1}{\sqrt{s}}, \frac{4}{\sqrt{s} - 2}\right\} \frac{\lambda (\sqrt{s} + 1)}{\phi^2(s,s)}.
\end{equation*}
This, together with \eqref{eq-thm-RB-1} and $\bar{x}_{\mathcal{N}^c} = 0$, one has that
\[
\|x^*-\bar{x}\|_2^2= \|x_{\mathcal{N}}^*-\bar{x}_{\mathcal{N}}\|_2^2+\|x_{\mathcal{N}^c}^*\|_2^2 \le \left(4+\max\left\{\frac{\sqrt{s} + 1}{\sqrt{s}}, \frac{4}{\sqrt{s} - 2}\right\}^2\right) \frac{\lambda^2 (\sqrt{s} + 1)^2}{\phi^4(s,s)},
\]
which establishes \eqref{eq-RB}, and the proof is complete.
\end{proof}

Theorem \ref{thm-OI+RB} is an important theoretical result in that it provides the oracle inequality \eqref{eq-OI} and the recovery bound \eqref{eq-RB} for the $\ell_{1\text{-}2}$ regularization problem \eqref{eq-L1-L2}. In particular, the oracle inequality \eqref{eq-OI} provides an upper bound on the square error of the linear system plus the violation of the nonzero components of $x^*$, and the recovery bound \eqref{eq-RB} provides an upper bound on the distance from $x^*$ to the ground true solution $\bar x$.

\begin{remark}\label{rem-OI+RB}
One can observe from \eqref{eq-RB} that the global optimal solution $x^*(\lambda)$ of the $\ell_{1\text{-}2}$ regularization problem \eqref{eq-L1-L2} has a recovery bound:
\begin{equation*}
\|x^*(\lambda)-\bar{x}\|_2^2\le \mathcal{O}\left(\lambda^2 (\sqrt{s} + 1)^2\right).
\end{equation*}
Particularly, when $s\ge 16$, it follows from \eqref{eq-RB} that
\begin{equation*}
\|x^*(\lambda)-\bar{x}\|_2^2\le \frac{8}{\phi^4(s,s)} \lambda^2 (\sqrt{s} + 1)^2.
\end{equation*}
\end{remark}

\section{Iterative Thresholding Algorithms}


The type of iterative thresholding algorithms (ITA) is one of the most popular and efficient numerical algorithms, with a simple formulation and a low computational complexity, for solving (convex or nonconvex) sparse optimization problems; see \cite{Blumensath09,Daubechies04,FoucartRauhut2013,HuJMLR17,XuZB12} and references therein. Particularly, the ITA for solving the $\ell_{1\text{-}2}$ regularization problem \eqref{eq-L1-L2} was proposed in \cite{LouYan2018}. However, limited by the nonconvexity of the $\ell_{1\text{-}2}$ regularization, the convergence theory of its ITA is still far from satisfactory: only convergence to a stationary point of problem \eqref{eq-L1-L2} was established in \cite[Theorem 1]{LouYan2018}; while there is still no theoretical evidence to guarantee the convergence to a global minimum or the ground true sparse solution.

To fill this typical gap of nonconvex optimization algorithms, this section aims to propose two novel ITA-type algorithms for solving the $\ell_{1\text{-}2}$ regularization problem \eqref{eq-L1-L2} and investigate their convergence to the ground true sparse solution.

\subsection{$\ell_{1\text{-}2}$ thresholding operator}

Proximal gradient algorithm (PGA) \cite{BeckTeboulle09, HuJMLR17} is one of the most popular and practical algorithms for composite optimization problem
\begin{equation}\label{eq-NCP}
\min_{x\in \R^n}\; f(x) + \varphi(x),
\end{equation}
where $f: \R^n \to \R$ is a differentiable function and $\varphi: \R^n \to \R$ is of some special structures. The iterative formula of the PGA for solving \eqref{eq-NCP} is
\begin{equation}\label{eq-PGA}
x^{k+1}:= {\rm Prox}_{\varphi} (x^k-v \nabla f(x^k)).
\end{equation}
where the proximal mapping ${\rm Prox}_{\varphi}:\R^n \to \R^n$ of $\varphi$ is defined by
\[
{\rm Prox}_{\varphi}(y):= {\rm arg}\min_{x\in \R^n}\, \varphi(x)+\frac{1}{2}\|x-y\|_2^2 \quad\, \mbox{for each } y\in \R^n.
\]

The type of ITAs can be understood as the applications of the PGAs to solve the composite optimization problem \eqref{eq-NCP} with certain regularization term. For example, the iterative soft (resp., hard, half) thresholding algorithm can be understood as the PGA for solving the $\ell_1$ (resp., $\ell_0$, $\ell_{1/2}$) regularization problem; see \cite{Daubechies04,Blumensath09,XuZB12} respectively.

Particularly, when applied to problem \eqref{eq-L1-L2} (with $\frac12\|Ax-b\|^2$ and $\lambda (\|x\|_1- \|x\|_2)$ in place of $f(x)$ and $\varphi(x)$ in \eqref{eq-NCP}), the ITA for solving the $\ell_{1\text{-}2}$ regularization problem \eqref{eq-L1-L2} has the following iterative formula
\begin{equation}\label{eq-PGA-L12}
x^{k+1}:= {\rm Prox}_{v \lambda(\|x\|_1-\|x\|_2)} (x^k-v A^\top (Ax^k-b)).
\end{equation}
It is presented in \cite[Lemma 1]{LouYan2018} that ${\rm Prox}_{\lambda(\|x\|_1- \|x\|_2)}$ has the following analytical formula.
\begin{lemma}\label{lem-prox-L1-L2}
The solution $x^*\in {\rm Prox}_{\lambda(\|x\|_1- \|x\|_2)} (y)$ has the following formulas.
\begin{enumerate}[{\rm (i)}]
 \item If $\|y\|_\infty>\lambda$, then $x^*= \left(1+\frac{\lambda}{\|z\|_2}\right) z $ with $z:=(|y|-\lambda)_+ \odot {\rm sign}(y)$.
 \item If $\|y\|_\infty=\lambda$, then $\|x^*\|_2=\lambda$ with $x_i^*y_i\ge 0$ for each $i\in [n]$ and $x_i^*=0$ whenever $|y_i|<\lambda$. 
 \item If $\|y\|_\infty<\lambda$, then $x^*$ is an $1$-sparse vector satisfying $\|x^*\|_2=\|y\|_\infty$ with $x_i^*y_i\ge 0$ for each $i\in [n]$ and $x_i^*=0$ whenever $|y_i|<\|y\|_\infty$. 
\end{enumerate}
\end{lemma}

Associated to the consistency theory (cf. Theorem \ref{thm-OI+RB}), the regularization parameter $\lambda$ should always be set to be quite small in numerical experiments, so as to guarantee the perfect recovery and numerical performance. In this case when $\lambda$ is set to be small, the situation in Lemma \ref{lem-prox-L1-L2} (i) always appears in the process of numerical computation. Hence to simplify the computation of proximal mapping of the $\ell_{1\text{-}2}$ regularization, we define an $\ell_{1\text{-}2}$ thresholding operator
 by
\begin{equation}\label{eq-threshold}
\mathbb{T}_{\lambda}(x) := \mathbb{E}_{\lambda} \circ \mathbb{S}_{\lambda}(x),
\end{equation}
where $\mathbb{S}_{\lambda}$ and $\mathbb{E}_{\lambda}$ are the soft-thresholding operator and the enlargement operator defined respectively by
\begin{equation}\label{eq-soft}
\mathbb{S}_{\lambda}(x) := (|x|-\lambda)_+ \odot {\rm sign}(x) \quad \mbox{and} \quad \mathbb{E}_{\lambda}(x) := \left(1+\frac{\lambda}{\|x\|_2}\right) x.
\end{equation}
It is obtained from Lemma \ref{lem-prox-L1-L2}(i) and \eqref{eq-threshold} that
\begin{equation}\label{eq-ITA=PGA}
\mathbb{T}_{\lambda}(y) = {\rm Prox}_{\lambda(\|x\|_1- \|x\|_2)} (y) \quad \mbox{whenever } \|y\|_\infty>\lambda.
\end{equation}

Below we discuss some useful properties of the $\ell_{1\text{-}2}$ thresholding operator $\mathbb{T}_{\lambda}$.
%

\begin{lemma}\label{lem-LTO}
The $\ell_{1\text{-}2}$ thresholding operator $\mathbb{T}_{\lambda}:\R^n \rightarrow \R^n$ satisfies the following properties:
\begin{enumerate}[{\rm (i)}]
 \item The thresholding property:
  \begin{equation}\label{eq-ThrO}
  (\mathbb{T}_{\lambda}(x))_i = 0 \quad \mbox{whenever } \, |x_i|\le \lambda.
  \end{equation}
 \item The shrinkage property:
  \begin{equation}\label{eq-LTO}
  \|\mathbb{T}_{\lambda}(x)-x \|_\infty \le \lambda.
  \end{equation}
\end{enumerate}
\end{lemma}
\begin{proof}
(i) It directly follows from \eqref{eq-soft} and \eqref{eq-threshold} that
\begin{equation*}
|x_i|\le \lambda \quad \Rightarrow \quad (\mathbb{S}_{\lambda}(x))_i = 0 \quad \Rightarrow \quad (\mathbb{T}_{\lambda}(x))_i = 0.
\end{equation*}

(ii) One can observe from \eqref{eq-ThrO} that
\begin{equation}\label{eq-lem-LTO-1}
|(\mathbb{T}_{\lambda}(x))_i-x_i| \le \lambda \quad \mbox{whenever } \, |x_i|\le \lambda.
\end{equation}
Then it remains to discuss the case when $|x_i|> \lambda$. In this case, it follows from \eqref{eq-soft} and \eqref{eq-threshold} that
\begin{equation}\label{eq-lem-LTO-2}
(\mathbb{S}_{\lambda}(x))_i = (|x_i|-\lambda) \, {\rm sign}(x_i) = x_i - \lambda \, {\rm sign}(x_i),
\end{equation}
and
\[
(\mathbb{T}_{\lambda}(x))_i = (\mathbb{E}_{\lambda} \circ \mathbb{S}_{\lambda}(x))_i = \left(1+\frac{\lambda}{\|\mathbb{S}_{\lambda}(x)\|_2}\right) (\mathbb{S}_{\lambda}(x))_i.
\]
Combining the above two inequalities, we obtain that
\begin{align}
(\mathbb{T}_{\lambda}(x))_i-x_i & = \left(1+\frac{\lambda}{\|\mathbb{S}_{\lambda}(x)\|_2}\right) (\mathbb{S}_{\lambda}(x))_i - \left((\mathbb{S}_{\lambda}(x))_i + \lambda \, {\rm sign}(x_i)\right) \notag \\
& = \frac{\lambda}{\|\mathbb{S}_{\lambda}(x)\|_2} (\mathbb{S}_{\lambda}(x))_i - \lambda \, {\rm sign}(x_i) \label{eq-lem-LTO-3} \\
& = \lambda \, {\rm sign}(x_i) \left( \frac{|(\mathbb{S}_{\lambda}(x))_i|}{\|\mathbb{S}_{\lambda}(x)\|_2} - 1\right) \notag
\end{align}
(due to the fact that $(\mathbb{S}_{\lambda}(x))_i$ and $x_i$ have the same sign). Noting that $1 - \frac{|(\mathbb{S}_{\lambda}(x))_i|}{\|\mathbb{S}_{\lambda}(x)\|_2}  \in [0, 1)$, it follows that $|(\mathbb{T}_{\lambda}(x))_i-x_i| < \lambda$. This, together with \eqref{eq-lem-LTO-1}, shows \eqref{eq-LTO}. 
\end{proof}

\begin{remark}\label{rem-LTO}
{\rm (i)} It is indicated from Lemma \ref{lem-LTO} that the $\ell_{1\text{-}2}$ thresholding operator 
$\mathbb{T}_{\lambda}$ inherits the thresholding and shrinkage properties, that are important for approaching a sparse solution.

{\rm (ii)} The enlargement operator $\mathbb{E}_{\lambda}$ enlarges the absolute values of nonzero components, hence the $\ell_{1\text{-}2}$ thresholding operator $\mathbb{T}_{\lambda}$ is able to overcome the over-penalization phenomenon of the soft thresholding operator $\mathbb{S}_{\lambda}$ on the components with large magnitude. In particular, we have by \eqref{eq-lem-LTO-2} and \eqref{eq-lem-LTO-3} that
\[
0< \frac{x_i-(\mathbb{T}_{\lambda}(x))_i}{x_i-(\mathbb{S}_{\lambda}(x))_i} = 1 - \frac{|(\mathbb{S}_{\lambda}(x))_i|}{\|\mathbb{S}_{\lambda}(x)\|_2} \le 1,
\]
and hence
\[
|(\mathbb{S}_{\lambda}(x))_i| \le |(\mathbb{T}_{\lambda}(x))_i| \le |x_i|.
\]
\end{remark}

Associated with the $\ell_{1\text{-}2}$ thresholding operator $\mathbb{T}_{\lambda}$, the standard ITA for solving the $\ell_{1\text{-}2}$ regularization problem \eqref{eq-L1-L2} has the following iterative form:
\begin{equation}\label{eq-ITA}
x^{k+1}:=\mathbb{T}_{\lambda}(x^k-v A^\top(Ax^k-b)).
\end{equation}
\begin{remark}
It is pointed out by Lemma \ref{lem-prox-L1-L2}(i), also by \eqref{eq-ITA=PGA}, that the thresholding operator $\mathbb{T}_{\lambda}$ equals to the proximal operator of the $\ell_{1\text{-}2}$ regularization whenever $\lambda$ is small. Consequently, when $\lambda$ is set to be small, the ITA \eqref{eq-ITA} is equivalent to the PGA \eqref{eq-PGA-L12} for the $\ell_{1\text{-}2}$ regularization problem \eqref{eq-L1-L2}.
\end{remark}
There is an open gap between the theoretical and algorithmic studies of the $\ell_{1\text{-}2}$ regularization problem \eqref{eq-L1-L2}. In particular, the consistency theory in Theorem \ref{thm-OI+RB} provides a recovery bound for the global minimum of problem \eqref{eq-L1-L2} to the ground true sparse solution; while the convergence theory in \cite[Theorem 1]{LouYan2018} only presents the convergence of the standard ITA \eqref{eq-ITA} to a stationary point of problem \eqref{eq-L1-L2}. However, there is still no theoretical evidence to guarantee the convergence to global minimum or the ground true sparse solution. This is also a typical gap of nonconvex optimization algorithms.

Throughout this paper, we assume that the observation $b$ is collected from a linear transform $A$ on a ground true $s$-sparse solution $\bar{x}$ with support $\mathcal{S}$ and an unknown noise $\varepsilon$:
\begin{equation}\label{eq-true-sol-e}
b=A\bar{x} + \varepsilon \, \mbox{ with } \, \mathcal{S}:=\{i:\bar{x}_{i}\neq 0\} \, \mbox{ and } \, s:=|\mathcal{S}|.
\end{equation}
In the remainder of this section, we aim to fill this gap by providing a positive theoretical guarantee for the sparsity recovery capability of the ITA-type algorithms for solving the $\ell_{1\text{-}2}$ regularization problem \eqref{eq-L1-L2}. In particular, we will propose two types of ITAs by combining with the truncation technique \cite{Blumensath09} and the continuation technique \cite{Yin08}, respectively, and investigate their convergence to the ground true sparse solution $\bar{x}$ of \eqref{eq-LS} under the assumption of the restricted isometry property (RIP) \cite{CandesTao05}.

\begin{remark}
The RIP \cite{CandesTao05} has been widely used for
the convergence analysis of sparse optimization algorithms \cite{Blumensath09, FoucartRauhut2013, NeedellCoSaMP2009, TZhang13}.
Many types of random matrices, including Gaussian, Bernoulli, and partial Fourier matrices, have been shown to satisfy the RIP with exponentially high probability \cite{BlanchardRIP2011}.
\end{remark}

%


We end this subsection with some useful properties of the RIP in the following lemma, which are taken from \cite[Lemma 1]{HuMP25}, also \cite[Proposition 3.1]{NeedellCoSaMP2009} and \cite[Lemmas 6.16 and 6.20]{FoucartRauhut2013}, and will be helpful in convergence analysis of our proposed algorithms.


\begin{lemma}\label{lem-RIP}
Suppose that $A$ satisfies the RIP with $\delta_s<1$. Let $x\in\mathbb{R}^n$, $\varepsilon\in\mathbb{R}^m$, $\mathcal{I}, \mathcal{J}\subseteq [n]$, and $v\ge 0$. Then the following assertions are true.
\begin{itemize}
	\item [\rm{(i)}] If $|\mathcal{I}\cup\operatorname{supp}(x)|\le s$, then $\|((\mathbb{I}-v A^{\top}A)x)_{\mathcal{I}}\|_2\le (|1-v| + v\delta_s)\|x\|_2$.
	\item [\rm{(ii)}] If $|\mathcal{I}|\le s$, then $		\|A^{\top}_{\mathcal{I}} \varepsilon\|_2\le\sqrt{1+\delta_s}\|\varepsilon\|_2$.
  \item [\rm{(iii)}] If $\mathcal{I} \cap \mathcal{J} = \emptyset$ and $|\mathcal{I}\cup \mathcal{J}|\le s$, then $
  \|A_{\mathcal{J}}^\top A_{\mathcal{I}} x_{\mathcal{I}}\|_2 \le \delta_s\|x_{\mathcal{I}}\|_2$.
\end{itemize}
\end{lemma}

\subsection{Iterative thresholding algorithm with truncation technique}

In the sequence generated by sparse optimization algorithms, the components with large magnitude are viewed as good approximation to the ground true sparse solution, while the small components can only be attributed to the presence of noise in observation. In order to avert the noise causing by the small components, the truncation operator, denoted by $\mathbb{H}_s$ that sets all but the largest $s$ elements of a vector (in magnitude) to zero,
is widely applied in sparse optimization algorithms to ensure the sparsity structure of the solution by discarding the small components; see \cite{Blumensath09, FoucartRauhut2013, ZhaoSIAM2020} and references therein.

By virtue of the standard $\ell_{1\text{-}2}$ thresholding operator \eqref{eq-threshold} and the truncation operator, we propose an iterative thresholding algorithm with the truncation technique (ITAT) to solve the $\ell_{1\text{-}2}$ regularization problem \eqref{eq-L1-L2} and approach the ground true sparse solution of \eqref{eq-LS}.


\begin{algorithm}[ITAT]\label{alg-ITAT}
Select a regularization parameter $\lambda>0$, a truncation parameter $s\in \IN$, and set a random initial point $x^0 \in \R^n$ and stepsize $v>0$.
For each $k\in \IN$, having $x^k$, we determine $x^{k+1}$ via
\begin{equation}\label{eq-ILTHard}
x^{k+1}:=\mathbb{H}_s \circ \mathbb{T}_{v\lambda}\left(x^k-v A^\top(Ax^k-b)\right).
\end{equation}
\end{algorithm}

\begin{remark}
{\rm (i)} 
The ITAT adopts the truncation operator $\mathbb{H}_s$ to maintain the sparsity level $s$ of the sequence $\{x^k\}$, which is helpful for guaranteeing the convergence to the ground true solution with the required sparsity of \eqref{eq-LS}; see Theorem \ref{thm-Hard}.

{\rm (ii)}
Since the truncation operator $\mathbb{H}_s$ and the $\ell_{1\text{-}2}$ thresholding operator $\mathbb{T}_{\lambda}$ are very simple to calculate, the ITAT inherits the significant advantages of simple formulation and low computational complexity, and thus is extremely efficient for large-scale problems.
\end{remark}

The following lemma recalls from \cite[Theorem 1]{ShenLi2018} a tight bound of truncation operator $\mathbb{H}_s(\cdot)$, which is useful in convergence analysis.

\begin{lemma}\label{lem-truncation}
Let $x\in \R^n$ be an $s$-sparse vector and $z\in \R^n$ be an arbitrary vector. Then
\[
\|\mathbb{H}_s(z)-x\|_2\le \frac{\sqrt{5}+1}2 \|z-x\|_2.
\]
\end{lemma}

The main result of this subsection is as follows, in which we establish the convergence of the ITAT to an approximate true sparse solution of \eqref{eq-LS} under the assumption of the RIP.


\begin{theorem}\label{thm-Hard}
Suppose that $A$ satisfies the $3s$-RIP with $\delta_{3s}<\frac{\sqrt{5}-1}2 \approx 0.618$, and let $\{x^k\}$ be a sequence generated by Algorithm \ref{alg-ITAT} with stepsize satisfying 
\begin{equation}\label{eq-thm-Hard-asp-RIP}
\frac{0.382}{1-\delta_{3s}}\approx \frac{3-\sqrt{5}}{2(1-\delta_{3s})} < v< \frac{\sqrt{5}+1}{2(1-\delta_{3s})} \approx \frac{1.618}{1-\delta_{3s}}.
\end{equation}
Then $\{x^k\}$ converges approximately to $\bar{x}$ at a geometric rate; particularly,
\begin{equation}\label{thm-Hard-eq1}
\|x^k - \bar{x}\|_2 \le \rho^{k}\|x^0 - \bar{x}\|_2+ \frac{v(\sqrt{5}+1)}{2(1-\rho)}(\sqrt{1+\delta_{2s}} \|\varepsilon\|_2 +\sqrt{2s} \lambda),
\end{equation}
where $\rho:= \frac{\sqrt{5}+1}2(|1-v| + v\delta_{3s}) \in (0,1)$.
\end{theorem}

\begin{proof}
To proceed the convergence analysis, we rewrite the process of Algorithm \ref{alg-ITAT} into the following three steps:
\begin{equation}\label{eq-thm-Hard-1}
  y^{k}:=x^k-vA^\top(Ax^k-b),\quad
  z^{k}:=\mathbb{T}_{v\lambda}(y^k),\quad
  x^{k+1}:=\mathbb{H}_s(z^k).
\end{equation}
Moreover, for the sake of simplicity, we write
\begin{equation}\label{eq-thm-Hard-1a}
r^k:=x^k-\bar{x}, \quad \mathcal{S}_k:={\rm supp}(x^k)\quad \mbox{and} \quad \mathcal{I}_{k}=\mathcal{S} \cup \mathcal{S}_{k};
\end{equation}
and then one observes from the $s$-sparsity of $\bar{x}$ and the third equality of \eqref{eq-thm-Hard-1} that
\begin{equation}\label{eq-thm-Hard-1aa}
|\mathcal{S}|=s, \quad |\mathcal{S}_k|\le s \quad \mbox{and} \quad |\mathcal{I}_{k}|\le 2s \quad \mbox{for each $k\in \IN$}.
\end{equation}

Noting by \eqref{eq-thm-Hard-1a} that
\begin{equation}\label{eq-thm-Hard-1b}
{\rm supp}(\bar{x}) = \mathcal{S} \subseteq \mathcal{I}_{k+1} \quad \mbox{and} \quad {\rm supp}(x^{k+1}) =\mathcal{S}_{k+1} \subseteq \mathcal{I}_{k+1},
\end{equation}
we get by \eqref{eq-thm-Hard-1} and by Lemma \ref{lem-truncation} (with $z^{k}_{\mathcal{I}_{k+1}}$ and $\bar{x}_{\mathcal{I}_{k+1}}$ in place of $z$ and $x$) that
\begin{align}
\|r^{k+1}\|_2 & = \|x^{k+1}_{\mathcal{I}_{k+1}} - \bar{x}_{\mathcal{I}_{k+1}}\|_2 \notag \\
& \le \frac{\sqrt{5}+1}2\|z^{k}_{\mathcal{I}_{k+1}}-\bar{x}_{\mathcal{I}_{k+1}}\|_2 \label{eq-thm-Hard-3} \\
& \le \frac{\sqrt{5}+1}2(\|\bar{x}_{\mathcal{I}_{k+1}}-y^{k}_{\mathcal{I}_{k+1}}\|_2+\|y^{k}_{\mathcal{I}_{k+1}}-z^{k}_{\mathcal{I}_{k+1}}\|_2). \notag
\end{align}
%

Noting by \eqref{eq-thm-Hard-1} that $z^{k}=\mathbb{T}_{v\lambda}(y^k)$, we have by Lemma \ref{lem-LTO} that $\|z^k-y^k\|_{\infty}\le v\lambda$. Combining this with \eqref{eq-thm-Hard-1aa}, we achieve by \eqref{eq-L1-2} that
\begin{equation}\label{eq-thm-Hard-3a}
\|y^{k}_{\mathcal{I}_{k+1}}-z^k_{\mathcal{I}_{k+1}}\|_2 \le \sqrt{|\mathcal{I}_{k+1}|}\|y^{k}-z^k\|_\infty \le \sqrt{2s} v\lambda.
\end{equation}
On the other hand, we obtain by the first equality of \eqref{eq-thm-Hard-1} and \eqref{eq-true-sol-e} that
\[
y^{k}=x^k- v A^\top(Ax^k-b) =x^k - v A^\top Ar^k+ v A^\top \varepsilon
\]
(due to \eqref{eq-thm-Hard-1a}); hence it follows that
\begin{equation}\label{eq-thm-Hard-4}
\|y^k_{\mathcal{I}_{k+1}} - \bar{x}_{\mathcal{I}_{k+1}}\|_2 \le \|((I-vA^\top A)r^k)_{\mathcal{I}_{k+1}}\|_2+v\|A_{\mathcal{I}_{k+1}}^\top \varepsilon\|_2.
\end{equation}
Note by \eqref{eq-thm-Hard-1b} that $|\mathcal{I}_{k+1}|\le 2s$ and $|\mathcal{I}_{k+1}\cup{\rm supp}(r^k)|=|\mathcal{I}_{k+1}\cup \mathcal{S}_k|\le 3s$. Then by the assumption of $3s$-RIP of $A$ and \eqref{eq-thm-Hard-asp-RIP}, we obtain by Lemmas \ref{lem-RIP}(i) and (ii) that
\[
\|((I-vA^\top A)r^k)_{\mathcal{I}_{k+1}}\|_2 \le (|1-v| + v\delta_{3s})\|r^k\|_2 \quad \mbox{and} \quad \|A_{\mathcal{I}_{k+1}}^\top \varepsilon\|_2 \le \sqrt{1+\delta_{2s}}\|\varepsilon\|_2,
\]
respectively. By the above two inequalities, \eqref{eq-thm-Hard-4} is reduced to
\begin{equation*}
\|y^k_{\mathcal{I}_{k+1}} - \bar{x}_{\mathcal{I}_{k+1}}\|_2\le (|1-v| + v\delta_{3s})\|r^k\|_2+v\sqrt{1+\delta_{2s}}\|\varepsilon\|_2.
\end{equation*}
This, together with \eqref{eq-thm-Hard-3} and \eqref{eq-thm-Hard-3a}, yields that
\begin{equation}\label{eq-thm-Hard-5}
\|r^{k+1}\|_2 \le \frac{\sqrt{5}+1}2(|1-v| + v\delta_{3s})\|r^k\|_2+\frac{\sqrt{5}+1}2v (\sqrt{1+\delta_{2s}} \|\varepsilon\|_2 +\sqrt{2s} \lambda)
\end{equation}
Let $\rho:=\frac{\sqrt{5}+1}2(|1-v| + v\delta_{3s})$. By the RIP assumption that $\delta_{3s}<\frac{\sqrt{5}-1}2$, we can check that $\rho \in (0,1)$ when either $\frac{3-\sqrt{5}}{2(1-\delta_{3s})} < v \le 1$ or $1\le v < \frac{\sqrt{5}+1}{2(1-\delta_{3s})}$, that is under the assumption \eqref{eq-thm-Hard-asp-RIP}. Then we obtain inductively by \eqref{eq-thm-Hard-5} that
\begin{align}
\| x^{k+1} - \bar{x}\|_2 &\le \rho\|x^k-\bar{x}\|_2+\frac{\sqrt{5}+1}2v (\sqrt{1+\delta_{2s}} \|\varepsilon\|_2 +\sqrt{2s} \lambda) \notag \\
&\le \cdots \notag \\
&\le \rho^{k+1}\|x^0-\bar{x}\|_2+\frac{v(\sqrt{5}+1)}{2(1-\rho)}(\sqrt{1+\delta_{2s}} \|\varepsilon\|_2 +\sqrt{2s} \lambda). \notag
\end{align}
The proof is complete.
\end{proof}

\begin{remark}\label{rem-ILSTAT}

{\rm (i)} Theorem \ref{thm-Hard} shows a geometric convergence rate of the ITAT to an approximate true sparse solution of \eqref{eq-LS} within a tolerance. The tolerance in \eqref{thm-Hard-eq1} has an additive form of a term on noise level $\mathcal{O}(\|\varepsilon\|_2)$ and a term on regularization parameter $\mathcal{O}(\sqrt{s}\lambda)$, that has the same order with the recovery bound in \eqref{eq-RB}. The term $\mathcal{O}(\sqrt{s}\lambda)$ could be small when a small regularization parameter $\lambda$ is selected, and it will be illustrated in our numerical experiments in Section \ref{sec-numer} that the best regularization parameter is about $\lambda = 10^{-3}$.

{\rm (ii)} Theorem \ref{thm-Hard} improves the convergence result in \cite[Theorem 2]{HuMP25} in the sense that it requires the weaker RIP and stepsize assumptions.
In particular, the assumptions that $\delta_{3s}< 0.618$ and \eqref{eq-thm-Hard-asp-RIP} are weaker than the ones that $\delta_{3s}< \frac12$ and $
\frac1{2(1-\delta_{3s})} < v\le \frac1{1-\delta_{3s}}$ assumed in \cite[Theorem 2]{HuMP25}.

{\rm (iii)} By \eqref{thm-Hard-eq1}, we obtain the complexity of the ITAT that
\begin{equation}\label{thm-Hard-complex}
\|x^{k^*} - \bar{x}\|_2 \le \frac{\sqrt{5}+3-2\rho}{1-\rho} v(\sqrt{1+\delta_{2s}} \|\varepsilon\|_2 +\sqrt{2s} \lambda),
\end{equation}
holds after at most $k^*:= \lceil \log_{\rho^{-1}} \frac{\|x^0 - \bar{x}\|}{v(\sqrt{1+\delta_{2s}} \|\varepsilon\|_2 +\sqrt{2s} \lambda)} \rceil $ iterates. Indeed, we have by definition of $k^*$ that
\begin{equation*}
\rho^{k^*}\|x^0 - \bar{x}\|_2 \le v(\sqrt{1+\delta_{2s}} \|\varepsilon\|_2 +\sqrt{2s} \lambda).
\end{equation*}
This, together with \eqref{thm-Hard-eq1}, implies \eqref{thm-Hard-complex}.
\end{remark}

%
%

\subsection{Iterative thresholding algorithm with continuation technique}

Note that the regularization parameter $\lambda$ plays an important role in the numerical performance of sparse optimization algorithms. According to the recovery bound theory, the regularization parameter $\lambda$ should be small to guarantee the better recovery; however, the computational mathematics theory and extensive numerical studies show that a too small parameter will result in the ill-posedness of the subproblems and the convergence is faster if the parameter is properly larger.
To inherit both advantages in theoretical and numerical aspects, the continuation technique is a commonly used acceleration strategy for sparse optimization algorithms, that
adopts a geometrically decreasing sequence of regularization parameters $\{\lambda_k\}$ starting at a large one in place of a fixed but small one; see, e.g., \cite{HaleYinZhang08, WenYin2010, TZhang13, Yin08}. 

Inspired by the idea of the $\ell_{1\text{-}2}$ thresholding operator \eqref{eq-threshold} and the continuation technique, we propose an iterative thresholding algorithm with the continuation technique (ITAC) to solve the $\ell_{1\text{-}2}$ regularization problem \eqref{eq-L1-L2}.

\begin{algorithm}[ITAC]\label{alg-ITAC}
Select the regularization parameters $\lambda_0>0$, $\lambda>0$ and a continuation parameter $\gamma\in (0,1)$, and set the initial point $x^0:=0$ and stepsize $v>0$.
For each $k\in \IN$, having $x^k$ and $\lambda_k$, if $\lambda_k<\lambda$, then it stops and outputs $x^*:=x^k$; otherwise, update the regularization parameter $\lambda_{k+1}:=\gamma \lambda_k$ and determine $x^{k+1}$ via
\begin{equation}\label{eq-ILTC}
x^{k+1}_i:=\mathbb{T}_{v\lambda_k}\left(x^k-vA^\top(Ax^k-b)\right). 
\end{equation}
\end{algorithm}

\begin{remark}
The only difference between ITAC and ITA is the use of a dynamical sequence of regularization parameters.
The ITAC inherits the significant advantages of simple formulation and low computational complexity, and thus is extremely efficient for large-scale sparse optimization problems.
\end{remark}

The main theorem of this subsection is as follows, which provides certain parameters setting (relevant to the RIP) in the ITAC to guarantee its convergence to an approximate true sparse solution of \eqref{eq-LS} within a tolerance proportional to a noise level. In addition, the support of the output of the ITAC has no false prediction and is exactly a subset of the support of the ground true sparse solution.



\begin{theorem}\label{thm-ITAC}
Suppose that $A$ satisfies the RIP with
\begin{equation}\label{asp-ROC}
(\sqrt{s}+1)\delta_{s+1}<1. 
\end{equation}
Let
\begin{equation}\label{eq-eta}
\eta \in \left(0, 1-(\sqrt{s}+1)\delta_{s+1}\right),
\end{equation}
and the stepsize $ v\in (0,1]$ and
\begin{equation}\label{eq-para}
\lambda_0\ge \frac{\|\bar{x}\|_2}{\sqrt{s}+1}, \quad \lambda:=\frac{\sqrt{1+\delta_s}}{\eta}\|\varepsilon\|_2 \quad \mbox{and} \quad \gamma \in \left[\frac{(\sqrt{s}+1)v\delta_{s+1}}{1- \eta} + 1 -v,1\right).
\end{equation}
Let Algorithm \ref{alg-ITAC} with these parameters output $x^*$. Then it holds that
\begin{equation}\label{eq-ILTC-res}
{\rm supp}(x^*) \subseteq \mathcal{S} \quad \mbox{and} \quad \|x^*-\bar{x}\|_2 \le \frac{(1-\eta) \sqrt{1+\delta_s}}{\eta \delta_{s+1}}\|\varepsilon\|_2.
\end{equation}
\end{theorem}

\begin{proof}
By assumption \eqref{asp-ROC}, one checks that $1-(\sqrt{s}+1)\delta_{s+1}>0$, and thus $\eta$ in \eqref{eq-eta} is well-defined. This yields that $\frac{(\sqrt{s}+1)v\delta_{s+1}}{1- \eta} + 1 -v<1$. Hence $\gamma$ in \eqref{eq-para} is well-defined, and so does Algorithm \ref{alg-ITAC} with these parameters.

To furniture the proof, let Algorithm \ref{alg-ITAC} generate a finite sequence $\{x^k\}_{k=0}^K$ and output $x^*=x^K$. Write
\begin{equation}\label{eq-thm-ITAC-rho}
\rho:=\frac{1- \eta}{\delta_{s+1}},
\end{equation}
and
\begin{equation}\label{eq-thm-ITAC-2}
y^k:=x^k-v A^\top (Ax^k-b), \quad \mathcal{S}_k:={\rm supp}(x^k) \quad \mbox{and} \quad r^k:=x^k-\bar{x}
\end{equation}
for each $k\in [K]$. We shall show by induction that the following inclusion and estimate hold for each $k\in [K]$:
\begin{equation}\label{eq-thm-ITAC-key}
\mathcal{S}_k\subseteq \mathcal{S} \quad \text{and} \quad \|r^k\|_2 \le \rho \lambda_k.
\end{equation}
By the initial selection that $x^0:=0$, one has that $\mathcal{S}_0=\emptyset \subseteq \mathcal{S}$. By assumption of $\lambda_0$ in \eqref{eq-para} and definition of $\rho$ in \eqref{eq-thm-ITAC-rho}, we obtain by \eqref{eq-eta} that
\[
\rho \lambda_0 \ge \frac{1-\eta}{\delta_{s+1}} \frac{\|\bar{x}\|_2}{\sqrt{s}+1} > \|\bar{x}\|_2 = \|r^0\|_2.
\]
It is shown that \eqref{eq-thm-ITAC-key} holds for $k = 0$.

Now suppose that \eqref{eq-thm-ITAC-key} holds for iterate $k$ ($<K$). Then by \eqref{eq-thm-ITAC-2} and \eqref{eq-true-sol-e}, we have
\begin{equation}\label{eq-thm-ITAC-3}
y^k=x^k-v A^\top (Ax^k-A\bar{x}-\varepsilon)=x^k-v A^\top A_{\mathcal{S}}r^k_{\mathcal{S}} + v A^\top \varepsilon,
\end{equation}
where the second equality follows from the hypothesis $\mathcal{S}_k \subseteq \mathcal{S}$ in \eqref{eq-thm-ITAC-key}.
Fix $i \in S^c$. It follows from the hypothesis $\mathcal{S}_k \subseteq \mathcal{S}$ in \eqref{eq-thm-ITAC-key} that $x_i^k=0$, and then \eqref{eq-thm-ITAC-3} is reduced to
\begin{equation}\label{eq-thm-ITAC-4}
|y_i^k|\le v |A_i^\top A_{\mathcal{S}}r^k_{\mathcal{S}}|+v |A_i^\top \varepsilon|.
\end{equation}
Since $\{i\} \cap \mathcal{S} =\emptyset$, we obtain by Lemmas \ref{lem-RIP}(iii) and (ii) that
\begin{equation*}\label{eq-thm-ITAC-5}
|A_i^\top A_{\mathcal{S}}r^k_{\mathcal{S}}|\le \delta_{s+1}\|(x^k-\bar{x})_{\mathcal{S}}\|_2 = \delta_{s+1} \|r^k\|_2 \quad \mbox{and} \quad |A_i^\top \varepsilon|\le \sqrt{1+\delta_1} \|\varepsilon\|_2\le \sqrt{1+\delta_s} \|\varepsilon\|_2
\end{equation*}
(due to $\mathcal{S}_k \subseteq \mathcal{S}$ and the nondecreasing property that $\delta_1 \le \delta_s$).
This, together with \eqref{eq-thm-ITAC-4}, yields that
\begin{equation}\label{eq-thm-ITAC-5b}
|y_i^k|\le v \delta_{s+1} \|r^k\|_2+v \sqrt{1+\delta_s}\|\varepsilon\|_2.
\end{equation}
By the stop criterion that $\lambda_k< \lambda$ and the definition of $\lambda$ in \eqref{eq-para}, one sees that $\lambda_k\ge \lambda=\frac{\sqrt{1+\delta_s}}{\eta}\|\varepsilon\|_2$, that is, $\|\varepsilon\|_2\le \frac{\eta\lambda_k}{\sqrt{1+\delta_s}}$. This, together with hypothesis \eqref{eq-thm-ITAC-key}, deduces \eqref{eq-thm-ITAC-5b} to
\begin{equation}\label{eq-thm-ITAC-6}
|y_i^k|\le v \delta_{s+1} \|r^k\|_2+ v \sqrt{1+\delta_s} \|\varepsilon\|_2 \le v\delta_{s+1} \rho \lambda_k+ v\eta \lambda_k= v\lambda_k,
\end{equation}
where the equality holds by definition of $\rho$ in \eqref{eq-thm-ITAC-rho}.
Hence it follows from \eqref{eq-ThrO} that $x_i^{k+1}=0$; consequently, $i\in \mathcal{S}_{k+1}^c$.
This holds for any $i\in\mathcal{S}^c$. Then 
we get $\mathcal{S}^c \subseteq \mathcal{S}_{k+1}^c$, and equivalently, $\mathcal{S}_{k+1} \subseteq \mathcal{S}$.

On the other hand, we get by the inclusion $\mathcal{S}_{k+1} \subseteq \mathcal{S}$ that
\begin{equation}\label{eq-thm-ITAC-7a0}
\|x^{k+1}-\bar{x}\|_2=\|x^{k+1}_{\mathcal{S}}-\bar{x}_{\mathcal{S}}\|_2\le \|x^{k+1}_{\mathcal{S}}-y^{k}_{\mathcal{S}}\|_2 + \|y^{k}_{\mathcal{S}}-\bar{x}_{\mathcal{S}}\|_2.
\end{equation}
By \eqref{eq-L1-2} and in view of Algorithm \ref{alg-ITAC} that $x^{k+1}=\mathbb{T}_{v\lambda_k}(y^k)$, we obtain that
\begin{equation}
\|x^{k+1}_{\mathcal{S}}-y^{k}_{\mathcal{S}}\|_2\le \sqrt{s} \|x^{k+1}-y^{k}\|_{\infty} = \sqrt{s} \|\mathbb{T}_{v\lambda_k}(y^k)-y^{k}\|_{\infty} \le v \sqrt{s} \lambda_k,
\end{equation}
(by Lemma \ref{lem-LTO}), and by \eqref{eq-thm-ITAC-3} that
\begin{equation}\label{eq-thm-ITAC-7a}
\|y^{k}_{\mathcal{S}}-\bar{x}_{\mathcal{S}}\|_2 = \|((I-vA^\top A)r^{k})_{\mathcal{S}} + v A_{\mathcal{S}}^\top \varepsilon\|_2
\le \|((I-vA^\top A)r^{k})_{\mathcal{S}}\|_2 + v \|A_{\mathcal{S}}^\top \varepsilon\|_2.
\end{equation}
Noting $v\le 1$, it follows from Lemmas \ref{lem-RIP}(i) and (ii) that
\[
\|((I-vA^\top A)r^{k})_{\mathcal{S}}\|_2 \le (1-v+v\delta_s) \|r^k\|_2 \quad \mbox{and} \quad \|A_{\mathcal{S}}^\top \varepsilon\|_2 \le \sqrt{1+\delta_s} \|\varepsilon\|_2,
\]
respectively.
These, together with \eqref{eq-thm-ITAC-7a0}-\eqref{eq-thm-ITAC-7a}, implies that
\begin{equation}\label{eq-thm-ITAC-8a}
\|x^{k+1}-\bar{x}\|_2\le v \sqrt{s} \lambda_k + (1-v+v\delta_s) \|r^k\|_2 + v\sqrt{1+\delta_s} \|\varepsilon\|_2
\end{equation}
By the fact that $\delta_s \le \delta_{s+1}$ and by \eqref{eq-thm-ITAC-6}, one has that
\[
v\delta_s \|r^k\|_2 + v\sqrt{1+\delta_s} \|\varepsilon\|_2 \le v \delta_{s+1} \|r^k\|_2+ v \sqrt{1+\delta_s}\|\varepsilon\|_2 \le v \lambda_k.
\]
Consequently by \eqref{eq-thm-ITAC-key}, \eqref{eq-thm-ITAC-8a} is reduced to
\begin{equation}\label{eq-thm-ITAC-8b}
\|x^{k+1}-\bar{x}\|_2\le \left(\frac{v (\sqrt{s} +1)}{\rho} + 1-v\right)\rho \lambda_k.
\end{equation}
Noting by definition of $\rho$ in \eqref{eq-thm-ITAC-rho} that
\[
\frac{v (\sqrt{s} +1)}{\rho} + 1-v = \frac{(\sqrt{s} +1)v\delta_{s+1}}{1-\eta}+ 1 -v \le \gamma
\]
(due to definition of $\gamma$ in \eqref{eq-para}), \eqref{eq-thm-ITAC-8b} is reduced to $\|x^{k+1}-\bar{x}\|\le \rho \gamma \lambda_k = \rho \lambda_{k+1}$.
This, together with $\mathcal{S}_{k+1} \subseteq \mathcal{S}$, shows that \eqref{eq-thm-ITAC-key} holds for each iterate $k\in [K]$.
Then we conclude by \eqref{eq-thm-ITAC-key} that ${\rm supp}(x^*)\subseteq \mathcal{S}$ and
\[
\|x^*-\bar{x}\|_2\le \rho \lambda_K< \rho \lambda=\frac{(1- \eta)\sqrt{1+\delta_s}}{\delta_{s+1}\eta}\|\varepsilon\|_2
\]
by definitions of $\lambda$ and $\rho$ in \eqref{eq-para} and \eqref{eq-thm-ITAC-rho}. The proof is complete.
\end{proof}

%
%

\begin{remark}
\cite{Yulingjiao17} showed the convergence of the ISTA with the continuation technique (ISTAC) for the $\ell_1$ regularization problem to an approximate true sparse solution of \eqref{eq-LS} under the assumption of the MIP. Theorem \ref{thm-ITAC} extends and improves \cite[Theorem 2]{Yulingjiao17} in several aspects:
\begin{itemize}
 \item[-] Theorem \ref{thm-ITAC} extends \cite[Theorem 2]{Yulingjiao17} in the sense that (a) the $\ell_{1\text{-}2}$ regularization problem \eqref{eq-L1-L2} is an extension of the $\ell_1$ regularization problem considered in \cite{Yulingjiao17}; and (b) Theorem \ref{thm-ITAC} considers the case when $A$ is a general sensing matrix and the stepsize $v\in (0,1]$ that is an extension of \cite{Yulingjiao17} which assumed that $A$ is column normalized and $v=1$.
 \item[-] Theorem \ref{thm-ITAC} improves \cite[Theorem 2]{Yulingjiao17} in the sense that it requires the assumption of the RIP, which is weaker than the MIP assumed in \cite[Theorem 2]{Yulingjiao17}. 
\end{itemize}
\end{remark}

\section{Numerical experiments}\label{sec-numer}
This section aims to conduct numerical experiments of the ITAT and the ITAC for solving the $\ell_{1\text{-}2}$ regularization problem \eqref{eq-L1-L2}, by comparing them with several popular algorithms for
sparse optimization. All numerical experiments are implemented in R (3.5.2) and MATLAB R2018a on a personal desktop (Intel Core Duo i7-8700, 3.20GHz, 32.00 GB of RAM).

The simulation data are generated as follows. The sensing matrix $A\in\mathbb{R}^{m\times n}$ is generated as two types of random matrices:
\begin{enumerate}
  \item[(i)] The random Gaussian matrix, which is defined as
    \[
    A_i \backsim \mathcal{N}(0, \mathbb{I}_m/m).
    \]
  \item[(ii)] The random partial discrete cosine transform (PDCT) matrix, that is
    \[
    A_i := \frac1{\sqrt{m}} \cos(2i\pi \xi),
    \]
where $\xi \in \R^m \backsim \mathcal{U}([0, 1]^m)$ are uniformly and independently sampled from $[0,1]$.
\end{enumerate}
The random Gaussian and PDCT matrices fit for comprehensive sensing, since they are incoherent and have small RIP constants with high probability \cite{BlanchardRIP2011, Candes06b}.


The ground true sparse solution $\bar{x}\in\mathbb{R}^{n}$ is generated via randomly picking $s$ of its components as nonzeros, whose entries are also randomly generated as i.i.d. Gaussian, while the remaining components are all set to be zeros. The observation matrix $b$ is generated by $b = A\bar{x} + \sigma \, \epsilon$ with $\epsilon\in\mathbb{R}^{m}$ being an i.i.d. Gaussian noise. In numerical experiments, the problem size is set as $(m, n) = (256, 1024)$, and the standard deviation of the additive Gaussian noise is set as $\sigma=0.1\%$.

In the implementation of the ITAT and the ITAC, the stepsize and the initial point are set as $v_k \equiv 0.5$ and $X^0 = 0$ as default, and the stopping criterion is set as $\frac{\|x^k-x^{k-1}\|_2}{\|x^{k-1}\|_2}\leq 10^{-6}$ or the number of iterations is greater than 500. The performance of the algorithms is evaluated by:
\begin{itemize}
  \item[-] (Accuracy) The relative error from ground true solution (RE):
  \[
  \mbox{RE}:=\frac{\|x - \bar{x}\|_2}{\|\bar{x}\|_2}.
  \]
  \item[-] (Robustness) The ratio of successful recovery, in which $\mbox{RE}<0.01$, in 500 random trails.
\end{itemize}

The first experiment aims to show the numerical performance of the ITAC with different continuation parameter $\gamma$ and the ITAT with different truncation parameter $s$. In this experiment, the sparsity of the  true solution is set as $\bar{s} = 51$. Figures \ref{fig-para}(a) and (b) plot the average RE of the solution generated by the ITAC with $\gamma$ varying from $(0.9, 1)$ and the ITAT with $s$ varying from $(45, 70)$, respectively, in 500 random trials. It is demonstrated from Figure \ref{fig-para}(a) that the ITAC performs best in obtaining an accurate estimation when $\gamma = 0.98$. It is indicated from Figure \ref{fig-para}(b) that the ITAT cannot obtain an accurate estimation when $s< \bar{s}$ and approaches an accurate estimation when $s \ge \bar{s}$ slightly (within $20\%$).
Therefore, in the following numerical experiments, we set the continuation parameter $\gamma = 0.98$ and the truncation parameter $s= \bar{s}$ in ITAC and ITAT, respectively.

\begin{figure}[ht]
\centering
\subfigure[$\gamma$ in ITAC.]{\includegraphics[width=6.5cm]{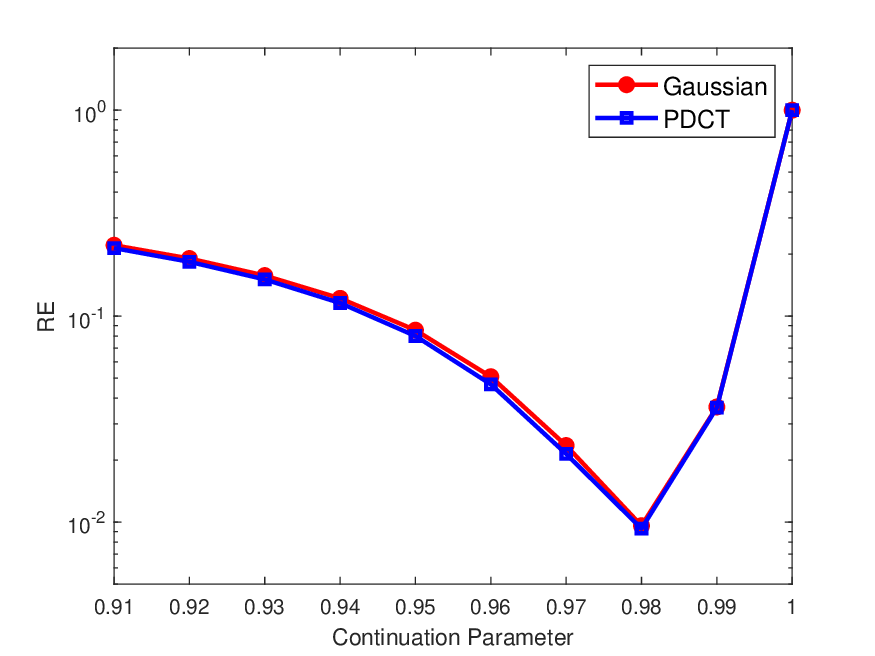}}\quad
\subfigure[$s$ in ITAT.]{\includegraphics[width=6.5cm]{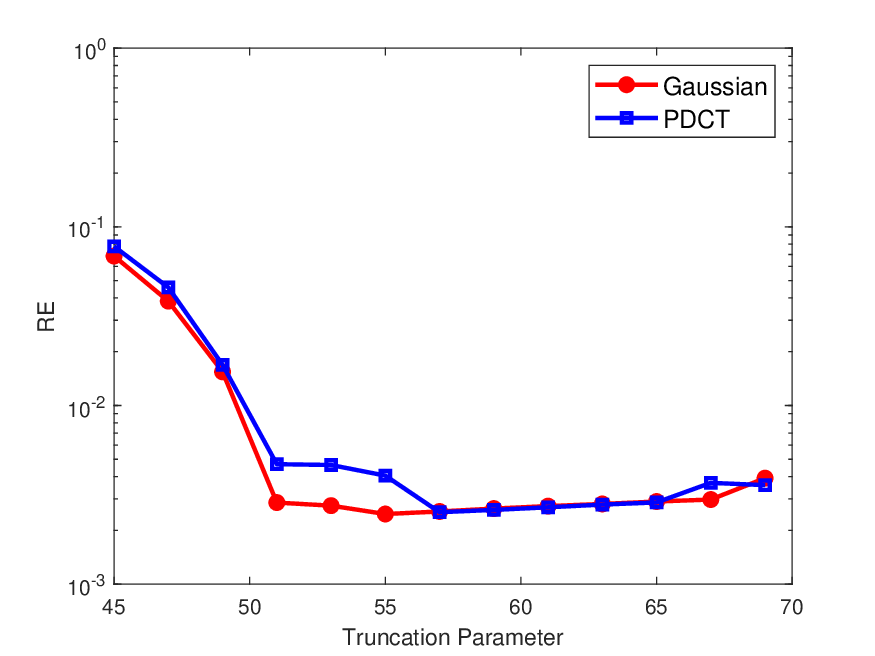}}
\caption{Numerical results of ITAC and ITAT with different parameters.}
\label{fig-para}
\end{figure}

The second experiment aims to compare the convergence behavior of the ITAC and the ITAT with the standard ITA \cite{LouYan2018} for the $\ell_{1\text{-}2}$ regularization and the ISTA \cite{Daubechies04} for Lasso. In this experiment, the sparsity ratio of the true solution is set as $s/n := 5\%$. Figure \ref{fig-converge} plots the average RE of these algorithms in 500 random trials along the number of iterations for Gaussian and PDCT matrices, respectively. It is displayed from Figure \ref{fig-converge} that the ITAC and the ITAT converge faster and achieve a more accurate solution than the standard ITA and the ISTA. This validates the accelerating capability and the convergence to an approximate true solution of the continuation technique and the truncation technique in PGA for sparse optimization, as well as the advantage of the $\ell_{1\text{-}2}$ regularization over Lasso.

\begin{figure}[ht]
\centering
\mbox{
\subfigure[Gaussian]{\includegraphics[width=6.5cm]{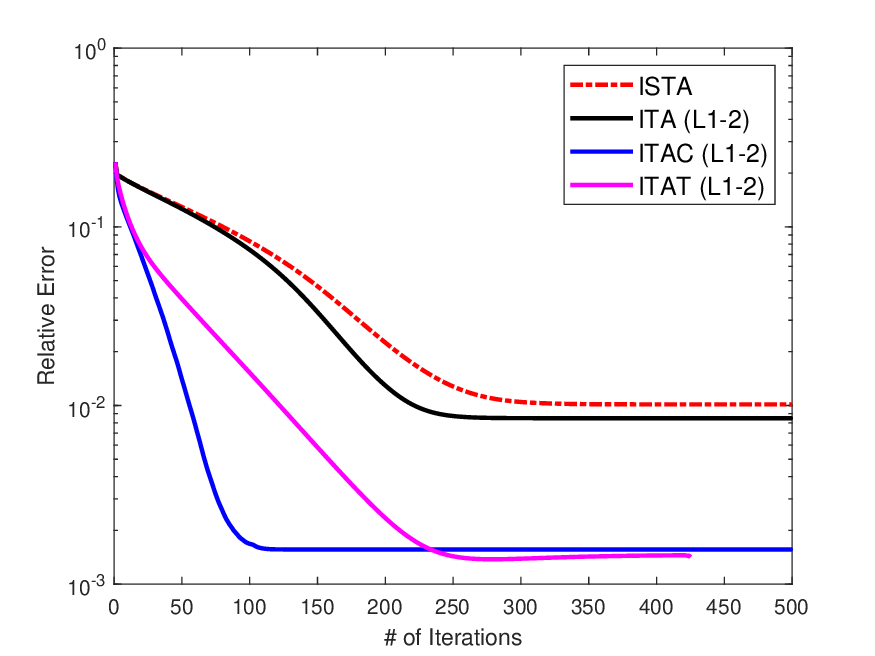}}\quad
\subfigure[PDCT]{\includegraphics[width=6.5cm]{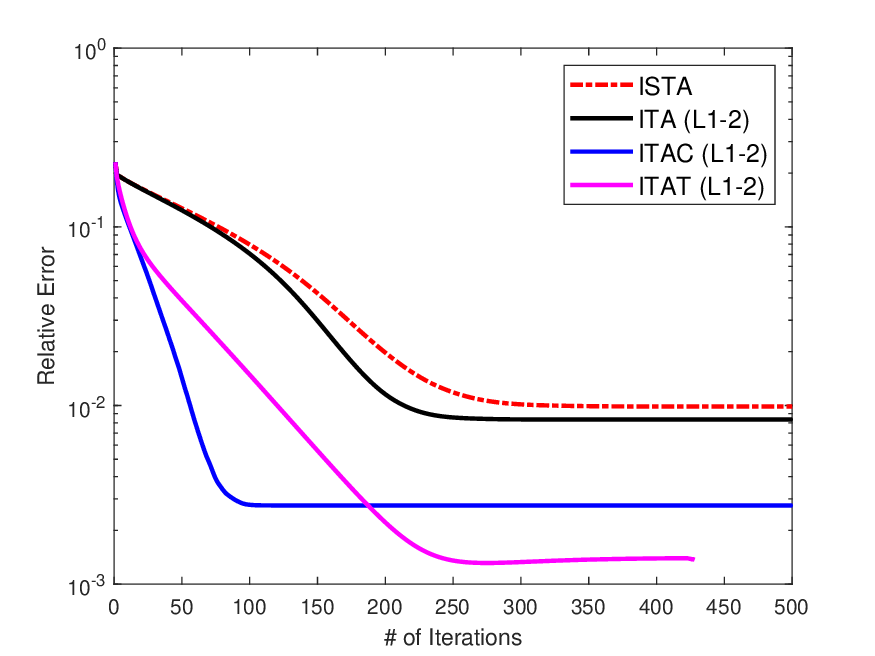}}}
\caption{Convergence behavior of ITAs for sparse optimization.}
\label{fig-converge}
\end{figure}

The third experiment aims to compare the stability of the ITAC and the ITAT with the standard ITA \cite{LouYan2018} and the DCA \cite{YinDL2015} for the $\ell_{1\text{-}2}$ regularization, and the popular algorithms for Lasso: the ISAT \cite{Daubechies04}, the ADMM \cite{YZ11} and the spectral projected gradient (SPGL1) \cite{BergFriedlander2008}. Figure \ref{fig-robust} plots the successful recovery rates and Table \ref{tab-CPU} records the (averaged) CPU time (in seconds) of these algorithms within 500 random trials at each sparsity level for Gaussian and PDCT matrices, respectively. It is indicated from Figure \ref{fig-robust} that the ITAC and the ITAT can achieve a higher successful recovery rate than the standard ITA, as well as the popular algorithms for Lasso.
Although the DCA obtains a little higher successful recovery rate than the ITAC and the ITAT, it is shown in Table \ref{tab-CPU} that the DCA consumes much more CPU time than other algorithms, including the ITAC and the ITAT, because the DCA costs much more time in adopting the CVX software to solve the convex subproblems.

%

\begin{figure}[ht]
\centering
\mbox{
\subfigure[Gaussian]{\includegraphics[width=6.5cm]{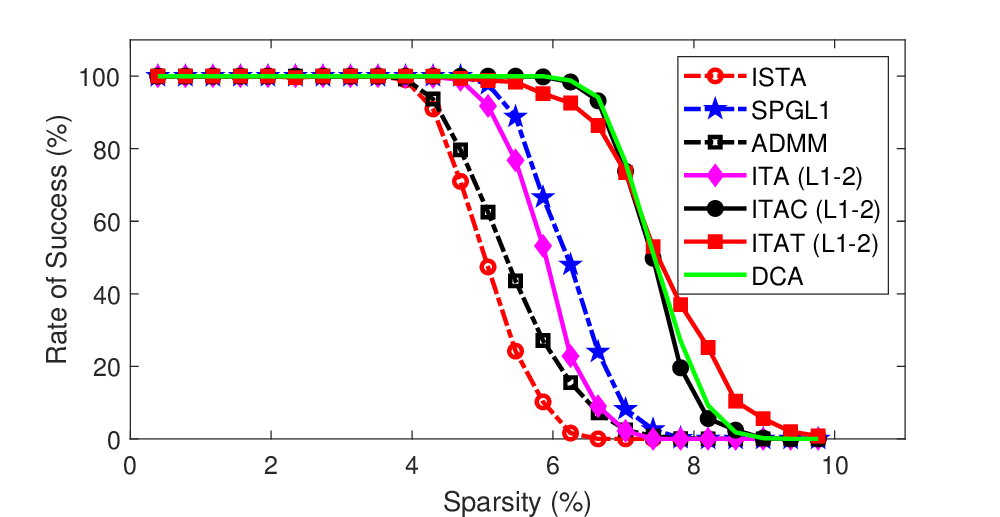}}\quad
\subfigure[PDCT]{\includegraphics[width=6.5cm]{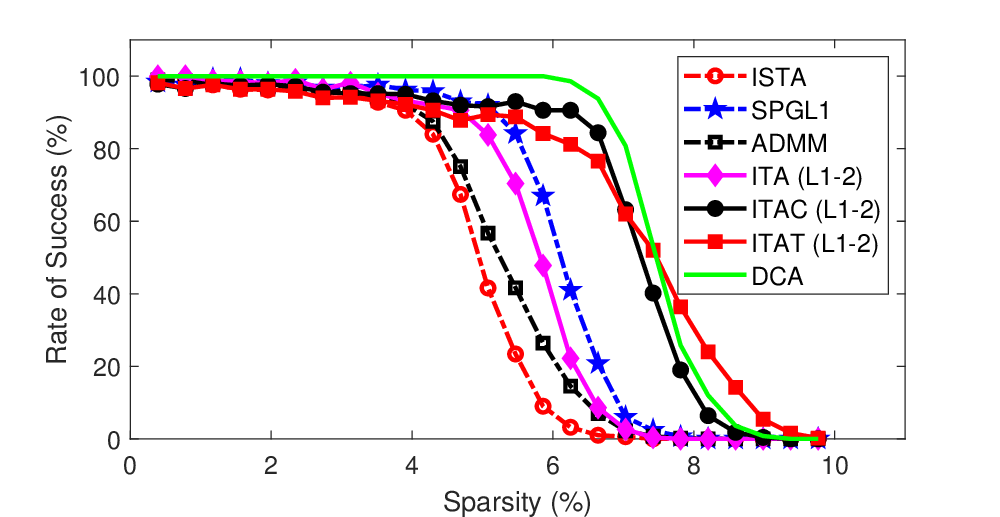}}}
\caption{Successful recovery rates of ITAs and algorithms for sparse optimization.}
\label{fig-robust}
\end{figure}

\begin{table}[h]
  \centering
  \caption{CPU time (in seconds) of ITAs and algorithms for sparse optimization.}
\begin{tabular}{c|ccc|cccc}
\hline
Problems    & \multicolumn{3}{|c|}{Lasso} &\multicolumn{4}{|c}{$\ell_{1\text{-}2}$ regularization} \\ \hline
Algorithms  & ITA       & SGPL1     & AMDD      & ITA       & ITAC      & ITAT      & DCA    \\ \hline
CPU time    & 0.34      & 0.81      &  0.07     & 0.35      & 0.28      & 0.34      & 2.34   \\ \hline
\end{tabular}
\label{tab-CPU}
\end{table}

In conclusion, preliminary numerical results show that the ITAC and the ITAT for the $\ell_{1\text{-}2}$ regularization have the strong sparsity promoting capability and outperforms the standard ITA for the $\ell_{1\text{-}2}$ regularization and popular algorithms for Lasso on both accuracy and robustness, benefiting from the $\ell_{1\text{-}2}$ penalty and the continuation or truncation technique.

\noindent \vskip 0.5cm
\begin{appendices}
\section[Appendix]{Proof of Proposition \ref{prop-REC-sufficient}}\label{App-JRFC}

The following lemma is useful to prove Proposition \ref{prop-REC-sufficient}, which is taken from \cite[Lemma 3.1]{Geer09}. 
\begin{lemma}\label{lem-RIP2}
Let $x, y\in \R^n$ and $\gamma \in (0,1)$ be such that $-\langle x,y\rangle\le \gamma\|x\|_2^2$. Then $(1-\gamma)\|x\|_2\le \|x+y\|_2$.
\end{lemma}

\begin{proof}[Proof of Proposition \ref{prop-REC-sufficient}]
Associated with the $\ell_{1\text{-}2}${\rm -REC}$(s,t)$, we define the feasible set
\begin{equation*}\label{eq-rec-RC-1}
C(s):=\{x\in R^n: \|x_{\mathcal{I}^c}\|_1\le \|x_{\mathcal{I}}\|_1+ \|x\|_2 \mbox{ for some } |\mathcal{I}|\le s\}.
\end{equation*}
Fix $x \in C(s)$. Then there exists $\mathcal{I}\subseteq [n]$ such that
\begin{equation}\label{eq-rec-RC-2}
|\mathcal{I}|\le s \quad \mbox{and}\quad \|x_{\mathcal{I}^c}\|_1\le \|x_{\mathcal{I}}\|_1+ \|x\|_2.
\end{equation}
Write $\mathcal{I}_k:=\mathcal{I}_k(x;t)$ as in \eqref{eq-Jk} and $\mathcal{J}:=\mathcal{I}\cup \mathcal{I}_0$ for the sake of simplicity. Hence
\begin{equation}\label{eq-rec-RC-2a}
\mbox{$\{\mathcal{J}, \mathcal{I}_1, \dots, \mathcal{I}_r\}$ is a disjoint decomposition of $[n]$ with $|\mathcal{I}_k|\le t$ and $|\mathcal{J}|\le s+t$}.
\end{equation}

(i) It follows from Lemma \ref{lem-bound-Nc} and \eqref{eq-rec-RC-2} that
\begin{equation}\label{eq-rec-RC-4}
\sum_{k=1}^r \|x_{\mathcal{I}_k}\|_2 \le \frac1{\sqrt{t}}\|x_{\mathcal{I}^c}\|_1
\le \frac1{\sqrt{t}} \|x_{\mathcal{I}}\|_1+\frac{1}{\sqrt{t}} \|x\|_2.
\end{equation}
Note by \eqref{eq-L1-2} (with $x_{\mathcal{I}}$ in place of $x$) that
\begin{equation}\label{eq-rec-RC-4a}
\|x_{\mathcal{I}}\|_1\le \sqrt{s} \|x_{\mathcal{I}}\|_2 \le \sqrt{s} \|x_{\mathcal{J}}\|_2,
\end{equation}
and by the subadditivity of $\ell_2$ norm and \eqref{eq-rec-RC-2a} that $\|x\|_2 \le \sum_{k=1}^r \|x_{\mathcal{I}_k}\|_2 + \|x_{\mathcal{J}}\|_2$.
Then \eqref{eq-rec-RC-4} is reduced to
\begin{equation}\label{eq-rec-RC-5}
(\sqrt{t}-1)\sum_{k=1}^r \|x_{\mathcal{I}_k}\|_2 \le (\sqrt{s}+1) \|x_{\mathcal{J}}\|_2.
\end{equation}
Noting by \eqref{eq-rec-RC-2a} that $|\mathcal{J}|\le s+t$ and $|\mathcal{I}_k|\le t$ for each $k\in \IN$, one has by \eqref{sparse-eigen} that
\[
\sigma_{\min}(s+t)\,\|x_{\mathcal{J}}\|_2 \le \|Ax_{\mathcal{J}}\|_2 \quad \mbox{and} \quad \|Ax_{\mathcal{I}_k}\|_2 \le \sigma_{\max}(t)\,\|x_{\mathcal{I}_k}\|_2.
\]
Due to the subadditivity of $\ell_2$ norm, this, together with \eqref{eq-rec-RC-5}, implies that
\begin{equation*}
\|Ax\|_2
\ge \|Ax_{\mathcal{J}}\|_2-\sum_{k=1}^{r}\|Ax_{\mathcal{I}_k}\|_2\\
\ge \left(\sigma_{\min}(s+t)- \frac{\sqrt{s}+1}{\sqrt{t}-1}\sigma_{\max}(t)\right)\|x_{\mathcal{J}}\|_2.
\end{equation*}
Since $x$ and $\mathcal{I}$ satisfying \eqref{eq-rec-RC-2} are arbitrary, it is shown to hold that
\[
\phi(s,t) \ge \sigma_{\min}(s+t)- \frac{\sqrt{s}+1}{\sqrt{t}-1}\sigma_{\max}(t) >0
\]
by assumption (i).

(ii) By Definition \ref{def-RIP}(ii) (cf. \eqref{eq-ROC}) and \eqref{eq-rec-RC-2a}, one has that
\[
|\langle Ax_{\mathcal{J}}, Ax_{\mathcal{J}^c}\rangle|
\le \sum_{k=1}^r|\langle Ax_{\mathcal{J}}, Ax_{\mathcal{I}_k}\rangle|\le \theta_{t,s+t} \|x_{\mathcal{J}}\|_2 \sum_{k=1}^r\|x_{\mathcal{I}_k}\|_2.
\]
Then it follows from \eqref{eq-rec-RC-5} that
\begin{equation}\label{eq-RIP-1}
|\langle Ax_{\mathcal{J}}, Ax_{\mathcal{J}^c}\rangle|
\le \theta_{t,s+t} \frac{\sqrt{s}+1}{\sqrt{t}-1}\|x_{\mathcal{J}}\|_2^2
\le \frac{ \theta_{t,s+t} (\sqrt{s}+1)}{(\sqrt{t}-1)(1-\delta_{s+t})} \|Ax_{\mathcal{J}}\|_2^2
\end{equation}
(by \eqref{eq-RIC}). One checks by assumption (ii) of this proposition that
\begin{equation}\label{eq-RIP-2}
0<\frac{\theta_{t,s+t}(\sqrt{s}+1)}{(\sqrt{t}-1)(1-\delta_{s+t})}<1.
\end{equation}
This, together with \eqref{eq-RIP-1}, shows that Lemma \ref{lem-RIP2} is applicable (with $Ax_{\mathcal{J}}$, $Ax_{\mathcal{J}^c}$, $\frac{\theta_{t,s+t}(\sqrt{s}+1)}{(\sqrt{t}-1)(1-\delta_{s+t})}$ in place of $x$, $y$, $\gamma$) to concluding that
\begin{equation*}
\|Ax\|_2
\ge \left(1-\frac{\theta_{t,s+t}(\sqrt{s}+1)}{(\sqrt{t}-1)(1-\delta_{s+t})}\right)\|Ax_{\mathcal{J}}\|_2\\
\ge \sqrt{1-\delta_{s+t}}\left(1-\frac{\theta_{t,s+t}(\sqrt{s}+1)}{(\sqrt{t}-1)(1-\delta_{s+t})}\right)\|x_{\mathcal{J}}\|_2
\end{equation*}
(due to \eqref{eq-RIC}).
Since $x$ and $I$ satisfying \eqref{eq-rec-RC-2} are arbitrary, we have by Definition \ref{def-REC} and \eqref{eq-RIP-2} that
\[
\phi(s,t)\ge \sqrt{1-\delta_{s+t}}\left(1-\frac{\theta_{t,s+t}(\sqrt{s}+1)}{(\sqrt{t}-1)(1-\delta_{s+t})}\right)>0.
\]

(iii) Note that
\begin{equation}\label{eq-RIP-3}
\|Ax\|_2^2 =\|Ax_{\mathcal{J}}\|_2^2+2\langle Ax_{\mathcal{J}}, Ax_{\mathcal{J}^c}\rangle+\|Ax_{\mathcal{J}^c}\|_2^2 \ge \|Ax_{\mathcal{J}}\|_2^2-2|\langle Ax_{\mathcal{J}}, Ax_{\mathcal{J}^c}\rangle|.
\end{equation}
In particular, we derive by the fact that $\mathcal{I} \subseteq \mathcal{J}$, \eqref{eq-rec-RC-2} and \eqref{eq-rec-RC-4a} that
\begin{equation}\label{eq-RIP-5a}
\|x_{\mathcal{J}^c}\|_1 \le \|x_{\mathcal{I}^c}\|_1 \le \|x_{\mathcal{I}}\|_1+ \|x\|_2 \le (\sqrt{s} + 1) \|x_{\mathcal{J}}\|_2+ \|x_{\mathcal{J}^c}\|_2.
\end{equation}
Noting by \eqref{eq-rec-RC-5} that
$\|x_{\mathcal{J}^c}\|_2\le \sum_{k=1}^r \|x_{\mathcal{I}_k}\|_2 \le \frac{\sqrt{s}+1}{\sqrt{t}-1} \|x_{\mathcal{J}}\|_2$,
\eqref{eq-RIP-5a} is reduced to
\begin{equation}\label{eq-RIP-5b}
\|x_{\mathcal{J}^c}\|_1 \le \frac{\sqrt{t} (\sqrt{s}+1)}{\sqrt{t}-1} \|x_{\mathcal{J}}\|_2.
\end{equation}
Moreover, one has by \eqref{eq-ROC} that
$\|A_{\mathcal{J}^c}^\top A x_{\mathcal{J}}\|_{\infty} \le \theta_{s+t,1} \|x_{\mathcal{J}}\|_2$.
By the above two inequalities, we achieve that
\begin{equation}\label{eq-RIP-5n}
|\langle Ax_{\mathcal{J}}, Ax_{\mathcal{J}^c}\rangle|\le \|x_{\mathcal{J}^c}\|_1 \|A_{\mathcal{J}^c}^\top A x_{\mathcal{J}}\|_{\infty} \le \theta_{s+t,1} \frac{\sqrt{t} (\sqrt{s}+1)}{\sqrt{t}-1} \|x_{\mathcal{J}}\|_2^2.
\end{equation}
On the other hand, we obtain by \eqref{sparse-eigen} that
\begin{equation}\label{eq-RIP-4}
\|Ax_{\mathcal{J}}\|_2^2 \ge \sigma_{\min}(s+t) \| x_{\mathcal{J}}\|_2^2.
\end{equation}
By \eqref{eq-RIP-5n} and \eqref{eq-RIP-4}, \eqref{eq-RIP-3} implies that
\[
\|Ax\|_2^2\ge \left(\sigma_{\min}(s+t) - 2\theta_{s+t,1} \frac{\sqrt{t} (\sqrt{s}+1)}{\sqrt{t}-1} \right) \| x_{\mathcal{J}}\|_2^2.
\]
Since $x$ and $\mathcal{I}$ satisfying \eqref{eq-rec-RC-2} are arbitrary, it is shown by assumption (iii) that
\[
\phi(s,t) \ge \left(\sigma_{\min}(s+t) - 2\theta_{s+t,1} \frac{\sqrt{t} (\sqrt{s}+1)}{\sqrt{t}-1} \right)^{\frac12} >0.
\]

(iv) Note by the mutual incoherence \eqref{eq-MIP} that
\begin{equation}\label{eq-MIP-2}
|\langle Ax_{\mathcal{J}}, Ax_{\mathcal{J}^c}\rangle| \le \mu \sum_{i\in \mathcal{J}^c} \sum_{j\in \mathcal{J}} |x_ix_j| \le \mu \|x_{\mathcal{J}}\|_1\|x_{\mathcal{J}^c}\|_1
\le \mu \frac{\sqrt{t(s+t)} (\sqrt{s}+1)}{\sqrt{t}-1} \|x_{\mathcal{J}}\|_2^2.
\end{equation}
(by \eqref{eq-L1-2} and \eqref{eq-RIP-5b}).
Combining this with \eqref{eq-RIP-3} and \eqref{eq-RIP-4}, we obtain that
\[
\| Ax\|_2^2\ge\left(\sigma_{\min}(s+t) - 2\mu \frac{\sqrt{t(s+t)} (\sqrt{s}+1)}{\sqrt{t}-1} \right)\| x_{\mathcal{J}}\|_2^2.
\]
Since $x$ and $\mathcal{I}$ satisfying \eqref{eq-rec-RC-2} are arbitrary, it is shown by assumption (iv) that
\[
\phi(s,t) \ge \left(\sigma_{\min}(s+t) - 2\mu \frac{\sqrt{t(s+t)} (\sqrt{s}+1)}{\sqrt{t}-1} \right)^{\frac12} >0.
\]

(v) Separating the diagonal and off-diagonal terms of $x_{\mathcal{J}}^{\top}A^{\top}Ax_{\mathcal{J}}$, we have that
\begin{equation}\label{eq-MIP-2a}
\|Ax_{\mathcal{J}}\|_2^2=\sum_{i\in \mathcal{J}} \langle A_i x_i, A_i x_i \rangle + \sum_{i,j\in \mathcal{J}: i\neq j} \langle A_i x_i, A_j x_j \rangle.
\end{equation}
By the assumption that $\|A_i\|=1$ for each $i\in [n]$, one has
$\sum_{i\in \mathcal{J}} \langle A_i x_i, A_i x_i \rangle = \|x_{\mathcal{J}}\|_2^2$.
By the mutual incoherence \eqref{eq-MIP}, we have that
\[
\sum_{i,j\in \mathcal{J}: i\neq j} \langle A_i x_i, A_j x_j \rangle \ge - \mu \sum_{i,j\in \mathcal{J}: i\neq j} |x_ix_j| \ge -\mu\|x_{\mathcal{J}}\|_1^2 \ge -(s+t)\mu \|x_{\mathcal{J}}\|_2^2.
\]
Consequently, \eqref{eq-MIP-2a} is reduced to $\|Ax_{\mathcal{J}}\|_2^2 \ge (1-(s+t)\mu)\|x_{\mathcal{J}}\|_2^2$. This, together with \eqref{eq-RIP-3} and \eqref{eq-MIP-2}, shows that
\[
\| Ax\|_2^2\ge\left(1 - (s+t)\mu - 2\mu \frac{\sqrt{t(s+t)} (\sqrt{s}+1)}{\sqrt{t}-1} \right)\| x_{\mathcal{J}}\|_2^2.
\]
Since $x$ and $\mathcal{I}$ satisfying \eqref{eq-rec-RC-2} are arbitrary, it holds that
\[
\phi(s,t) \ge \left( 1 - (s+t)\mu - 2\mu \frac{\sqrt{t(s+t)} (\sqrt{s}+1)}{\sqrt{t}-1} \right)^{\frac12} >0
\]
by assumption (v). The proof is complete.
\end{proof}
\end{appendices}

\addcontentsline{toc}{chapter}{Bibliography}
\bibliographystyle{abbrv}
\bibliography{thesis}

@article{HuMP25,
author = {Hu, Yaohua and Hu, Xinlin and Yang, Xiaoqi},
title = {On convergence of iterative thresholding algorithms to approximate sparse solution for composite nonconvex optimization},
journal = {Mathematical Programming},
volume = {211},
pages = {181--206},
year = {2025},
doi = {10.1007/s10107-024-02068-1},
}

@article{HuJMLR17,
author  = {Yaohua Hu and Chong Li and Kaiwen Meng and Jing Qin and Xiaoqi Yang},
title   = {Group Sparse Optimization via {$\ell_{p,q}$} Regularization},
journal = {Journal of Machine Learning Research},
year    = {2017},
volume  = {18},
number  = {30},
pages   = {1-52},
}

@book{FoucartRauhut2013,
  title={A Mathematical Introduction to Compressive Sensing},
  author={S. Foucart and H. Rauhut},
  year={2013},
  publisher={Springer},
  address={New York}
}

@article {Candes06b,
author = {Cand\`{e}s, Emmanuel  and Romberg, Justin K. and Tao, Terence},
title = {Stable signal recovery from incomplete and inaccurate measurements},
journal = {Communications on Pure and Applied Mathematics},
volume = {59},
number = {8},
pages = {1207-1223},
year = {2006},
}

@article{Fan01,
    author = {Fan, Jianqing and Li, Runze},
    journal = {Journal of the American Statistical Association},
    number = {456},
    pages = {1348--1360},
    title = {Variable Selection via Nonconcave Penalized Likelihood and Its Oracle Properties},
    volume = {96},
    year = {2001}
}

@article{BeckTeboulle09,
author = {Beck, A. and Teboulle, M.},
title = {A Fast Iterative Shrinkage-Thresholding Algorithm for Linear Inverse Problems},
journal = {SIAM Journal on Imaging Sciences},
volume = {2},
number = {1},
pages = {183-202},
year = {2009},
}

@article{Blumensath09,
title = "Iterative hard thresholding for compressed sensing",
journal = "Applied and Computational Harmonic Analysis",
volume = "27",
number = "3",
pages = "265-274",
year = "2009",
author = "Thomas Blumensath and Mike E. Davies",
}

@article {Daubechies04,
author = {Daubechies, I. and Defrise, M. and De Mol, C.},
title = {An iterative thresholding algorithm for linear inverse problems with a sparsity constraint},
journal = {Communications on Pure and Applied Mathematics},
volume = {57},
pages = {1413--1457},
year = {2004},
}

@article {YZ11,
    AUTHOR = {Yang, Junfeng and Zhang, Yin},
     TITLE = {Alternating direction algorithms for $\ell_1$-problems in compressive sensing},
   JOURNAL = {SIAM Journal on Scientific Computing},
    VOLUME = {33},
      YEAR = {2011},
    NUMBER = {1},
     PAGES = {250--278},
}

@article{Natarajan95,
author = {Natarajan, B.},
title = {Sparse Approximate Solutions to Linear Systems},
journal = {SIAM Journal on Computing},
volume = {24},
number = {2},
pages = {227-234},
year = {1995},
}

@ARTICLE{XuZB12,
author={Zongben Xu and Xiangyu Chang and Fengmin Xu and Hai Zhang},
journal={IEEE Transactions on Neural Networks and Learning Systems},
title={${L}_{1/2}$ Regularization: {A} Thresholding Representation Theory and a Fast Solver},
year={2012},
volume={23},
pages={1013-1027},
}

@article{ChenXJ10,
author = {Chen, X. and Xu, F. and Ye, Y.},
title = {Lower Bound Theory of Nonzero Entries in Solutions of $\ell_2$-$\ell_p$ Minimization},
journal = {SIAM Journal on Scientific Computing},
volume = {32},
number = {5},
pages = {2832-2852},
year = {2010},
}

@article{TZhang13,
author = {Xiao, L. and Zhang, T.},
title = {A Proximal-Gradient Homotopy Method for the Sparse Least-Squares Problem},
journal = {SIAM Journal on Optimization},
volume = {23},
number = {2},
pages = {1062-1091},
year = {2013},
doi = {10.1137/120869997},
}

@article{Geer09,
author = {Sara A. van de Geer and Peter B\"{u}hlmann},
title = {{On the conditions used to prove oracle results for the {L}asso}},
journal = {Electronic Journal of Statistics},
volume = {3},
year = {2009},
pages = {1360--1392},
}

@article{Bunea07,
author = {Florentina Bunea and Alexandre Tsybakov and Marten Wegkamp},
title = {{Sparsity oracle inequalities for the Lasso}},
journal = {Electronic Journal of Statistics},
volume = {1},
year = {2007},
pages = {169--194},
}

@ARTICLE{Bickel09,
    author = {Peter J. Bickel and Ya\'{a}cov Ritov and Alexandre B. Tsybakov},
    title = {Simultaneous analysis of {L}asso and {D}antzig selector},
    journal = {Annals of Statistics},
    year = {2009},
    volume = {37},
    pages = {1705--1732}
}

@article{Meinshausen09,
author = {Nicolai Meinshausen and Bin Yu},
title = {{L}asso-type recovery of sparse representations for high-dimensional data},
journal = {Annals of Statistics},
volume = {37},
year = {2009},
pages = { 246--270},
}

@article{CandesTao05,
author = {Emmanuel Cand\`{e}s and Terence Tao},
title = {Decoding by Linear Programming},
journal={IEEE Transactions on Information Theory},
volume = {51},
year = {2005},
pages = {4203--4215},
}

@article{HaleYinZhang08,
author = {Elaine T. Hale and Wotao Yin and Yin Zhang},
title = {Fixed-Point Continuation for {$\ell_1$}-Minimization: {M}ethodology and Convergence},
journal = {SIAM Journal on Optimization},
volume = {19},
number = {3},
pages = {1107--1130},
year = {2008},
doi = {10.1137/070698920},
}

@article{NeedellCoSaMP2009,
title = "{CoSaMP}: {I}terative signal recovery from incomplete and inaccurate samples",
journal = "Applied and Computational Harmonic Analysis",
volume = "26",
number = "3",
pages = "301--321",
year = "2009",
author = "D. Needell and J.A. Tropp",
}

@Article{Wright2015,
author="Wright, Stephen J.",
title="Coordinate descent algorithms",
journal="Mathematical Programming",
year="2015",
volume="151",
number="1",
pages="3--34",
issn="1436-4646",
doi="10.1007/s10107-015-0892-3",
}

@article{LouYan2018,
  title={Fast {L1-L2} Minimization via a Proximal Operator},
  author={Yifei Lou and Ming Yan},
  journal={Journal of Scientific Computing},
  volume={74},
  pages={767--785},
  year={2018},
}

@article{Donoho2006Most,
  title={For most large underdetermined systems of linear equations the minimal $\ell_1$-norm solution is also the sparsest solution},
  author={Donoho, D. L.},
  journal={Communications on Pure and Applied Mathematics},
  volume={59},
  number={6},
  pages={797--829},
  year={2006},
  publisher={Wiley Online Library}
}

@article{Donoho2001Uncertainty,
  title={Uncertainty principles and ideal atomic decomposition},
  author={Donoho, D. L and Huo, X. M.},
  journal={{IEEE} Transactions on Information Theory},
  volume={47},
  number={7},
  pages={2845--2862},
  year={2001}
}

@article{BlanchardRIP2011,
author = {Blanchard, Jeffrey D. and Cartis, Coralia and Tanner, Jared},
title = {Compressed Sensing: How Sharp Is the Restricted Isometry Property?},
journal = {SIAM Review},
volume = {53},
number = {1},
pages = {105-125},
year = {2011},
}

@article{Yin08,
author = {Yin, W. and Osher, S. and Goldfarb, D.},
title = {Bregman iterative algorithms for {$\ell_1$}-minimization with application to compressed sensing},
journal = {SIAM Journal on Imaging Sciences},
volume = {1},
number = {1},
pages = {143--168},
year = {2008},
}

@article{WenYin2010,
author = {Wen, Zaiwen and Yin, Wotao and Goldfarb, Donald and Zhang, Yin},
title = {A Fast Algorithm for Sparse Reconstruction Based on Shrinkage, Subspace Optimization, and Continuation},
journal = {SIAM Journal on Scientific Computing},
volume = {32},
number = {4},
pages = {1832--1857},
year = {2010},
doi = {10.1137/090747695},
}

@article{Yulingjiao17,
author = {Yuling Jiao and Bangti Jin and Xiliang Lu},
title = {Iterative Soft/Hard Thresholding with Homotopy Continuation for Sparse Recovery},
journal = {IEEE Signal Processing Letters},
volume = {24},
number = {6},
pages = {784--788},
year = {2017},
}

@article{ZhaoSIAM2020,
author = {Zhao, Yun-Bin},
title = {Optimal {$k$}-Thresholding Algorithms for Sparse Optimization Problems},
journal = {SIAM Journal on Optimization},
volume = {30},
number = {1},
pages = {31-55},
year = {2020},
doi = {10.1137/18M1219187},
}

@article{ZhangNSP2021,
  author={Zhang, Jing and Zhang, Shuguang},
  journal={IEEE Signal Processing Letters},
  title={Null Space Property of $\ell_{1-2}$ Minimization With Prior Support Information},
  year={2021},
  volume={28},
  number={},
  pages={1779--1783},
  doi={10.1109/LSP.2021.3106809}
}

@article{YinDL2015,
author = {Yin, Penghang and Lou, Yifei and He, Qi and Xin, Jack},
title = {Minimization of $\ell_{1-2}$ for Compressed Sensing},
journal = {SIAM Journal on Scientific Computing},
volume = {37},
number = {1},
pages = {A536-A563},
year = {2015},
doi = {10.1137/140952363},
}

@article{Esser2013,
author = {Esser, E. and Lou, Y. and Xin, J.},
title = {A Method for Finding Structured Sparse Solutions to Nonnegative Least Squares Problems with Applications},
journal = {SIAM Journal on Imaging Sciences},
volume = {6},
number = {4},
pages = {2010-2046},
year = {2013},
doi = {10.1137/13090540X},
}

@article{ZhangMCP2010,
  title={Nearly unbiased variable selection under minimax concave penalty},
  author={Zhang, C.-H.},
  journal={Annals of Statistics},
  volume={38},
  number={2},
  pages={894--942},
  year={2010},
}

@article{LouJSC2015,
author = {Lou, Y. and Yin, P. and He, Q and Xin, Jack},
title = {Computing Sparse Representation in a Highly Coherent Dictionary Based on Difference of $L_1$ and $L_2$},
journal = {Journal of Scientific Computing },
volume = {64},
pages = {178--196},
year = {2015},
doi = {10.1007/s10915-014-9930-1},
}

@article{GeNSP2018,
  author={Ge, Huanmin and Wen, Jinming and Chen, Wengu},
  journal={IEEE Signal Processing Letters},
  title={The Null Space Property of the Truncated $\ell_{1-2}$-Minimization},
  year={2018},
  volume={25},
  number={8},
  pages={1261--1265},
  doi={10.1109/LSP.2018.2852138}
}

@article{LouJSC2016,
  author={Yifei Lou and Penghang Yin and Jack Xin},
  journal={Journal of Scientific Computing},
  title={Point Source Super-resolution Via Non-convex $L_1$ Based Methods},
  year={2016},
  volume={68},
  pages={1082-1100},
}

@article{LouSIAM2015,
author = {Lou, Yifei and Zeng, Tieyong and Osher, Stanley and Xin, Jack},
title = {A Weighted Difference of Anisotropic and Isotropic Total Variation Model for Image Processing},
journal = {SIAM Journal on Imaging Sciences},
volume = {8},
number = {3},
pages = {1798-1823},
year = {2015},
doi = {10.1137/14098435X},
}

@article{MaLouSIAM2017,
author = {Ma, Tian-Hui and Lou, Yifei and Huang, Ting-Zhu},
title = {Truncated $l_{1-2}$ Models for Sparse Recovery and Rank Minimization},
journal = {SIAM Journal on Imaging Sciences},
volume = {10},
number = {3},
pages = {1346-1380},
year = {2017},
doi = {10.1137/16M1098929},
}

@article{BergFriedlander2008,
  Author = {E. van den Berg and M. P. Friedlander},
  Title = {Probing the Pareto frontier for basis pursuit solutions},
  year = {2008},
  journal = {SIAM Journal on Scientific Computing},
  volume = {31},
  number = {2},
  pages = {890--912},
}

@article{ShenLi2018,
  Author = {Jie Shen and Ping Li},
  Title = {A tight bound of hard thresholding},
  year = {2018},
  journal = {Journal of Machine Learning Research},
  volume = {18},
  pages = {1--42},
}
\end{document}